\documentclass[11pt]{article}
\usepackage{amsmath}
\usepackage{graphicx}
\usepackage{subcaption}
\usepackage{float}

\usepackage{textcomp}
\usepackage{amsbsy}
\usepackage{latexsym}
\usepackage[mathscr]{eucal}
\usepackage{amsfonts}
\usepackage{amssymb, amsthm}
\usepackage{mathtools}
\usepackage[usenames]{color}
\usepackage[dvipsnames]{xcolor}

\usepackage{tikz}
\usepackage{fullpage}

\usepackage[hypertexnames=false]{hyperref} 

\newcommand*\samethanks[1][\value{footnote}]{\footnotemark[#1]}

\title{Nonlinear approximation spaces for inverse problems}

\author{Albert Cohen\thanks{Laboratoire Jacques-Louis Lions, Sorbonne Universit\'e, 
4 place Jussieu, 75005 Paris, France (albert.cohen@sorbonne-universite.fr, matthieu.dolbeault@sorbonne-universite.fr and agustin.somacal@sorbonne-universite.fr)}, Matthieu Dolbeault\samethanks[1], Olga Mula\thanks{Technical University of Eindhoven, Netherlands
(o.mula@tue.nl)}, Agustin Somacal\samethanks[1]}
\date{\today}

\newcommand\eps{\varepsilon}
\newcommand\e{\varepsilon}
\newcommand\dist{\text{dist}}
\newcommand\diff{\text{diff}}
\renewcommand\t{\tilde}

\newcommand\<{\langle}
\renewcommand\>{\rangle}
\renewcommand\leq{\leqslant}
\renewcommand\geq{\geqslant}

\def\cM{{\cal M}}
\def\cN{{\cal N}}
\def\cP{{\cal P}}
\def\cT{{\cal T}}

\def\cF{{\cal F}}
\def\cC{{\cal C}}

\def\cK{{\cal K}}

\def\o{\overline}
\def\nl{\newline}
\def\vp{\varphi}
\def\({\left (}
\def\){\right )}

\newcommand{\R}{\mathbb R}
\newcommand{\N}{\mathbb N}
\newcommand{\Pp}{\mathbb P}
\def\Chi{\raise .3ex\hbox{\large $\chi$}}

\newcommand{\be}{\begin{equation}}
\newcommand{\ee}{\end{equation}}
\newcommand{\iref}[1]{{\rm (\ref{#1})}}

\newtheorem{theorem}{Theorem}[section]
\newtheorem{proposition}[theorem]{Proposition}
\newtheorem{lemma}[theorem]{Lemma}
\newtheorem{remark}[theorem]{Remark}
\newtheorem{definition}[theorem]{Definition}

\DeclareMathOperator{\argmin}{argmin}

\begin{document}

\maketitle

\begin{abstract}
This paper is concerned with the ubiquitous inverse problem of recovering an
unknown function $u$ from finitely many measurements 
possibly affected by noise. In recent years, inversion methods based
on {\it linear} approximation spaces were introduced in
\cite{BCDDPW2017,MPPY2015} with certified recovery bounds. 
It is however known that linear spaces become ineffective 
for approximating simple and relevant families of functions, such as piecewise smooth 
functions that typically occur in hyperbolic PDEs (shocks) or images (edges).
For such families, {\it nonlinear} spaces \cite{Devore1998} are known to significantly improve
the approximation performance. The first contribution of this paper is to provide with certified recovery 
bounds for inversion procedures based on nonlinear approximation spaces. The second contribution 
is the application of this framework
to the recovery of general bidimensional shapes from cell-average data. 
We also discuss how the application of our results to $n$-term approximation relates to classical results in compressed sensing.
\end{abstract}

\section{Introduction}

\subsection{The recovery problem}
In this paper, we treat the following state estimation problem in a general Banach space $V$. We want to recover an approximation to an unknown function $u\in V$ from data given by $m$ observations
\be
\label{measurements}
z_i \coloneqq \ell_i(u) + \eta_i, \quad i=1,\dots, m,
\ee
where $\ell_i:V \mapsto \R$ are known measurement functionals, and $\eta_i$ is additive noise. The functionals $\ell_i$ often correspond to the response of a physical measurement device but they can have a different interpretation depending on the application. Their behavior can be linear (in which case the $\ell_i$ are linear functionals from $V'$, the dual of $V$) or nonlinear.
This type of recovery problem is clearly ill-posed when the dimension of $V$ exceeds $m$. It arises ubiquitously in sampling 
and inverse problem applications where $V$ is infinite dimensional (to name a few, see \cite{AHP2013, ABCGMM2018, GLM2022, HG2010}). 

One natural strategy to address this difficulty is to search for a 
recovery of $u$ by an element of a low-dimensional reconstruction space $V_n\subset V$. The space $V_n$ could be either an $n$-dimensional 
linear subspace, or more generally a nonlinear approximation space parametrized by $n$ degrees of freedom, with $n\leq m$.

In order to obtain quantitative results for such recovery procedures, it is necessary to possess additional information about $u$, 
usually as an assumption that $u$ belongs to a certain model class $\cK$ contained in $V$. 
The approximation space $V_n$ is chosen in order to collectively approximate the elements of $\cK$ as well as possible, in the sense that
\[
\dist(\cK, V_n)_V \coloneqq \max_{u\in \cK} \min_{v\in V_n} ||u - v||_V
\]
is as small as possible for moderate values of $n$.

Numerous theoretical results and numerical algorithms have been proposed in several fields to study and solve the above recovery problem (below we recall some relevant results). However, to the best of our knowledge, they all involve at least one or several of the following assumptions:
\begin{itemize}
\item The $\ell_i$ are \emph{linear} functionals,
\item $V_n$ is a \emph{linear} (or affine) subspace of $V$,
\item $V$ is a \emph{Hilbert} space,
\item The model class $\cK$ is a \emph{ball in a smoothness space}, e.g., a unit ball in Lipschitz, Sobolev, or Besov spaces. Results involving this type of model classes have been intensively studied in the field of optimal recovery (see \cite{Bojanov1994, MR1977, NW2008}).
\end{itemize}
The goal of this paper is to develop and analyze inversion procedures that do not require any of the above assumptions. Our analysis and numerical algorithms can thus be applied to virtually any recovery problem. The starting point of our development is based on algorithms introduced for inverse state estimation using reduced order models of parametrized Partial Differential Equations (PDEs). We next recall the specific framework. The presentation will also serve to explain more in depth the motivations leading to propose the present generalization.

\subsection{State estimation with reduced models for parametrized PDE's}

A relevant scenario in inverse state estimation is when the model class $\cK$ is given by the set of solutions to some parameter-dependent PDE of the general form
\be
\label{pde}
\cP(u,y)=0,
\ee
where $\cP$ is a differential operator, $y$ a vector of parameters ranging in some domain $Y$ in $\R^d$,
and $u$ is the solution. If well-posedness holds in some Banach space $V$ for each $y\in Y$, we denote by $u(y)\in V$ the corresponding solution for the given parameter value $y$ and by
\[
\cM\coloneqq\{u(y) \;: \; y\in Y\},
\]
the {\it solution manifold}.

In inverse state estimation, we take $\cK=\cM$ for the model class so the unknown $u$ to recover belongs to $\cM$. However, the parameter $y$ that satisfies $u=u(y)$ is unknown, so we cannot solve the forward problem \eqref{pde} to approximate $u$. Instead, we must approximate $u$ from the partial observational data \eqref{measurements}, and the knowledge of the model class $\cK=\cM$.

For the manifold $\cM$, efficient approximation spaces $V_n$ are usually obtained by reduced modelling techniques. In their most simple format,
reduced models consist into linear spaces 
$(V_n)_{n\geq 0}$ with $\dim(V_n)=n$.
The ideal benchmark in this linear approximation setting is provided by the Kolmogorov $n$-width 
\[
d_n(\cM)_V\coloneqq\inf_{\dim(V_n)\leq n} {\rm dist}(\cM,V_n)_V,
\]
which describes the optimal approximation performance achievable by an $n$-dimensional space over the set $\cM$. 

Apart from very simplified cases, the space $V_n$ achieving the above infimum is usually out of reach. Practical model reduction techniques such as polynomial approximation in the parametrized domain \cite{CD2015acta,CDS2011,TWZ2017} or reduced bases \cite{EPR2010, HRS2015, MS2013, RHP2007, ZKA2019} construct spaces $V_n$ that are ``suboptimal yet good''. In particular, the reduced basis method, which generates $V_n$ by a specific selection of particular solution instances $u^1,\dots,u^n\in \cM$, has been
proved to have approximation error ${\rm dist}(\cM,V_n)_V$ that decays with the same polynomial or exponential
rates 
as $d_n(\cM)_V$, and in that sense are close to optimal \cite{DPW2013}.

\subsection{The PBDW method}

We take the {\it Parametrized Background Data Weak} (PBDW)
method as a starting point for our analysis. The PBDW method, first introduced in \cite{MPPY2015},
as well as several extensions, has  been the object of a series of works  \cite{BCDDPW2017, BCMN2018, CDDFMN2020, CDMN2022} on its optimality properties as a recovery algorithm.
It has also been used for different practical applications, see \cite{ABCGMM2018, GLM2022, HCBM2019}. We refer to \cite{Mula2022} for an overview of the state of the art on this approach, and its connections with different fields. For our current purposes, it will suffice to recall the first version of the algorithm, which is the goal of this section.

The PBDW method uses a linear approximation space $V_n$
of dimension $n\leq m$. Usually this space is a reduced model in applications. 
It is assumed that the $\ell_i$ are continuous linear functionals, that
is $\ell_i\in V'$, and that $V$ is a Hilbert space. Then, introducing the Riesz representers
$\omega_i\in V$ such that $\ell_i(v)=\<\omega_i,v\>_V$, the data of the noise-free observation
\[
\ell(u)\coloneqq(\ell_1(u),\dots,\ell_m(u)),
\]
is equivalent to that of the orthogonal projection $w=P_W u$ on the {\it Riesz measurement space}
\[
W\coloneqq{\rm span}\{\omega_1,\dots,\omega_m\}.
\] 
Assuming linear independence of the $\ell_i$, this space has dimension $m$. A critical
quantity is the number
\be
\mu=\mu(V_n,W)\coloneqq\max_{v\in V_n} \frac{\|v\|_V}{\|P_W v\|_V},
\label{mu}
\ee
that describes the ``stability'' of the description of an element of $V_n$ by
its projection onto $W$, and may be thought of as the inverse 
cosine of the angle between $W$ and $V_n$. In particular, this quantity
is finite only when $n\leq m$. It can be explicitly computed as the inverse of
the smallest singular value of a cross-grammian matrix between 
orthonormal bases of $V_n$ and $W$ (see \cite{BCDDPW2017, Mula2022}).

The PBDW method consists in solving the minimization problem
\[
\min_{v\in V_w}\min_{\t v\in V_n} \|v-\t v\|_V,
\]
where $V_w\coloneqq w+ W^\perp$ is the set of all states $v$ such that $P_W v=w$.
We denote by $(u^*,\t u)\in V_w\times V_n$ the minimizing pair,
which is unique when $\mu<\infty$, and can be computed by solving 
an $n\times n$ linear system.
The function $\t u$ may be seen as
a particular best-fit estimator of $u$ on $V_n$, since it is also defined by 
\[
\t u\coloneqq{\rm argmin} \{\|P_W v-w\|_V\;: \; v\in V_n\}.
\]
The function $u^*$ can be derived from $\t u$ by the correction procedure
\[
u^*\coloneqq\t u+ (w-P_W \t u),
\]
which shows that $u^*\in V_n+W$. It may be thought of as a generalized interpolation estimator, since it agrees
with the observed data ($P_W u^* = P_W u$). In the case of noise-free data, it is proved
in \cite{BCDDPW2017, MPPY2015} that these estimators satisfy the
recovery bounds
\[
\|u-\t u\|_V \leq \mu \min_{v\in V_n} \|u-v\|_V \quad{\rm and} \quad
\|u-u^*\|_V \leq \mu \min_{v\in V_n\oplus (V_n^\perp \cap W)} \|u-v\|_V.
\]
These bounds reflect a typical trade-off in the choice of the reduced basis space, since making $n$ larger
has both effect of decreasing the approximation error $\min_{v\in V_n} \|u-v\|_V $ and increasing 
the stability constant $\mu=\mu(V_n,W)$. 

When the PBDW method is applied to noisy data, amounting in 
observing a perturbed version $\o w$ of $w=P_W u$, the recovery bounds remain valid up
to the additional term $\mu \|w-\o w\|_V$. In summary, one has for both estimators
\be
\label{genest}
\max\{\|u-\t u\|_V, \|u-u^*\|_V \} \leq \mu(e_n(u)+\kappa),
\ee
where
\[
e_n(u)\coloneqq\min_{v\in V_n} \|u-v\|_V
\]
is the reduced model approximation error and $\kappa\coloneqq\|w-\o w\|_V$
is the noise error measured in the space $W$. Note that since the additive perturbations $\eta_i$ are
applied to the data $\ell_i(u)$, a natural model for the measurement noise is to assume a bound 
of the form
\be
\|\eta\|_p \leq \e,
\label{noisemodel}
\ee
for the vector $\eta=(\eta_1,\dots,\eta_m)$, typically in the max norm $p=\infty$ or euclidean norm $p=2$. Therefore,
one has $\kappa\leq \beta \e$, where
\[
\beta\coloneqq\max_{v\in W}\frac {\|v\|_V}{\|\ell(v)\|_p},
\]
resulting in a bound of the form $\mu e_n(u)+\mu\beta \e$ for both estimators.

\subsection{Towards nonlinear approximation spaces}

The simplicity of the PBDW method and its variants comes together with a 
fundamental limitation on its performance: it is by essence a linear reconstruction method
with recovery bounds tied to the approximation error $e_n(u)$. When the only prior information
is that the unknown function $u$ belongs to a class $\cK$, with $\cK=\cM$ the solution manifold in 
the case of parametric PDEs, its best performance over $\cK$ is thus limited by the $n$-width
$d_n(\cK)_V$ and in turn by $d_m(\cK)_V$ since $n\leq m$.

In several simple yet relevant settings, it is known that $n$-widths have poor decay with $n$.
One instance is when the class $\cK$ contains piecewise smooth states, with 
a state-dependent location of jump discontinuities.
As an elementary example, one can
easily check that if $V=L^2([0,1])$ and $\cK$ is the set all indicator functions $u=\Chi_{[a,b]}$
with $a,b\in [0,1]$, one has $d_n(\cK)_V \sim n^{-1/2}$. This decay is of course even slower
for more general classes of piecewise smooth function in higher dimension, see in particular \cite[Chapter 3, equation (3.76)]{BCOW2017}. 
Such functions are typical in parametrized hyperbolic PDEs, due to the presence of shocks with positions that 
differ when parameters entering the velocity vary. We refer to \cite{BBEELM2022, BCMN2018, ELMV2020, GK2019, OR2016, Welper2015} 
for other examples of parametric PDEs whose solution manifold has slow Kolmogorov $n$-width decay.

For such classes of functions, nonlinear approximation methods are well known to perform significantly
better than their linear counterparts. Typical representatives of such methods include approximation by rational fractions, 
free knot splines or adaptive finite elements, best $n$-term approximation in a basis or dictionary, 
neural network or various tensor formats. In these instances the space $V_n$ still depends on
$n$ or ${\cal O}(n)$ parameters but is not anymore a linear
space. We refer to \cite{Devore1998} for a general introduction on the topic of nonlinear approximation.

\subsection{Objective and outline}

The objective of this paper is to study the natural extensions of the PBDW method
to such nonlinear approximation spaces and identify the basic structural properties
that lead to near optimal recovery estimates similar to \iref{genest}. 

We begin in \S \ref{section2} by considering the most general setting where $V$ is a Banach space,
$V_n$ a nonlinear approximation family, and the $\ell_i$ are functionals defined on $V$
that are not necessarily linear, but Lipschitz continuous, that is
\be
\|\ell(v)-\ell(\t v)\|_Z\leq \alpha_Z\|v-\t v\|_V, \quad v,\t v\in V.
\label{stab1}
\ee
Here $\|\cdot\|_Z$ can be any given norm defined over $\R^m$ with the constant $\alpha_Z$
depending on this choice of norm. In this framework, we discuss the best-fit estimation procedure
that consists in minimizing the distance to the observed data in a given norm $\|\cdot\|_Z$.

Our main structural assumption on $V_n$ is the following {\it inverse stability property}:
the reduced model is stable with respect to the measurement functionals
if there exists a finite constant $\mu_Z$ such that 
\be
\|v-\t v\|_V \leq \mu_Z\|\ell(v)-\ell(\t v)\|_Z, \quad v,\t v\in V_n.
\label{invstab}
\ee
The stability constant $\mu_Z$ depends on the $Z$ norm 
and plays a role similar to that of $\mu$ in the linear case. In particular, we show 
that this constant is finite only if $n\leq m$. The resulting estimator $\t u$ is 
then proved to satisfy a general recovery bound of the form
\[
\|u-\t u\|_V \leq C_1e_n(u)+C_2\|\eta\|_p,
\]
where $e_n(u)\coloneqq\min_{v\in V_n} \|u-v\|_V$ is the nonlinear 
reduced model approximation error, $\|\eta\|_p$ the level of measurement noise
in $\ell^p$ norm, and the constants $C_1$ and $C_2$ depend on $\alpha_Z$ and $\mu_Z$.

In \S \ref{section3}, we consider the more particular setting where the $\ell_i$ are linear functionals.
Then, we show that constants $C_1$ and $C_2$
are each minimized by a different choice of norm $\|\cdot\|_Z$, resulting in two different
best fit estimators $\t u$, as already observed in \cite{BGM2017} in the case of linear reduced
models. This particular setting also allows us to introduce
a generalized interpolation estimator $u^*$ and establish similar recovery estimates
for $\|u-u^*\|_V$.

We next apply our framework to the inverse problem that consists in 
recovering a general shape $\Omega$, identified to its characteristic function $\Chi_\Omega$,
based on cell average data 
\[
a_T(\Omega)\coloneqq\frac 1{|T|}\int_{T} \Chi_\Omega,\quad T\in\cT,
\]
where $\cT$ is a fixed cartesian mesh. One motivation for this problem is the 
design of finite volume schemes for the computation of solutions to transport PDEs on 
such meshes. 

We first discuss in \S \ref{section4} the best estimation rate in terms of the mesh size $h$
that can be achieved by standard linear reconstructions,
and which is essentially that of piecewise constant approximations, that is
${\cal O}(h^{1/q})$ regardless of the smoothness of the boundary $\partial\Omega$.
This intrinsic limitation is due to the presence of the jump discontinuity 
that is not well resolved by the mesh.

We then discuss in \S \ref{section5} a local recovery strategy based 
on a nonlinear approximation space $V_n$ that consists
of characteristic functions of half-planes which can fit the boundary of $\Omega$ at 
a subcell resolution level, as already proposed in \cite{ACDD2005, ELVIRA, ELVIRAexactness, LVIRA}.
One main result, whose proof is given in an appendix, is that 
this approximation space is stable in the sense of \iref{invstab} 
with respect to cell average measurements on a stencil of $3\times 3$ squares.
In turn, if $\Omega$ has a $C^2$ boundary, the recovered shape $\t \Omega$ is proved to satisfy an estimate of the form
\[
\|\Chi_\Omega-\Chi_{\t \Omega}\|_{L^q} \leq Ch^{2/q},
\]
where $h$ is the mesh size, which cannot be achieved by any linear reconstruction. This paves the way to
higher order reconstruction methods for smoother boundaries by using local nonlinear approximation
spaces with curved boundaries and larger stencils.

Finally, we discuss in \S \ref{section6} the application of our results to the recovery of  large vectors of size $N$ from $m<N$ linear measurements, up to the error of best $n$-term approximation. This problem is well-known in compressed sensing \cite{CTR,FR}, and was in particular studied in \cite{CDD2009} which discusses the importance  of the recovery norm $\|\cdot\|_V$ to understand if near-optimal recovery bounds can be achieved with $m$ not much larger than $n$. We show that the structural assumptions identified in our general setting are naturally related to the so-called {\it null space property} introduced in \cite{CDD2009}.

\section{Nonlinear reduction of inverse problems}
\label{section2}

\subsection{A general framework}

In full generality we are interested in recovering functions $u$ in
a general Banach space $V$ with norm $\|\cdot\|_V$, from the measurement vector $z=(z_1,\dots,z_m)\in\R^m$
given by \iref{measurements}. A recovery (or inversion) map
\[
z \to R(z),
\]
takes this vector to an approximation $R(z)$ of $u$. We are interested in controlling the recovery error $\|u-R(z)\|_V$.

To build the recovery map $R$, we use a nonlinear approximation space of dimension $n$ is a family of functions that can be described by 
$n$ parameters. Loosely speaking, this means that there exists a set $S\subset \R^n$ and a 
continuous map $\vp: S\to V$ such that 
\[
V_n\coloneqq\{\vp(x)\;: \; x\in S\}.
\]
Note that this definition covers the case of an $n$ dimensional linear subspace since we can choose $S=\R^n$ and $\vp$ a linear map. 

Our main assumptions are the Lipschitz stability of the functionals $\ell_i$ over the whole space $V$
and their inverse Lipschitz stability over the nonlinear approximation space $V_n$, expressed by \iref{stab1}
and \iref{invstab}, respectively. Note that since $\R^m$ is finite dimensional, the norm $\|\cdot\|_Z$
that is chosen in $\R^m$ to express these properties could be arbitrary up to a modification 
of the stability constants $\alpha_Z,\mu_Z$. These constants can be optimally defined as
\[
\alpha_Z=\sup_{v_1,v_2\in V}\frac{\|\ell(v_1)-\ell(v_2)\|_Z}{\|v_1-v_2\|_V},
\]
and
\[
\mu_Z=\sup_{v_1,v_2\in V_n}\frac{\|v_1-v_2\|_V}{\|\ell(v_1)-\ell(v_2)\|_Z}.
\]
Note that one always has $\alpha_Z\mu_Z\geq 1$.

\begin{remark}
Note that when $V_n$ is an $n$-dimensional space and the $\ell_i$ are linear functionals,
the quantity $\mu_Z$ may be rewritten as
\[
\mu_Z=\max_{v\in V_n}\frac{\|v\|_V}{\|\ell(v)\|_Z}.
\]
As discussed further,
the quantity $\mu$ defined in \iref{mu} for the analysis of the PBDW method is
an instance of $\mu_Z$ corresponding to a particular choice of norm $\|\cdot\|_Z$.
Assuming the $\ell_i$ are independent functionals, one 
easily checks that finiteness of this quantity imposes that $n\leq m$. Indeed, if $n>m$, there exists
a non-trivial $v\in V_n\cap \cN$, where
\[
\cN\coloneqq\{v\;:\; \ell(v)=0\}
\]
is the null space of the measurement map that has codimension $m$, and therefore $\mu_Z$
is infinite. 
\end{remark}

\begin{remark}
The restriction $n\leq m$ is also needed for nonlinear spaces $V_n$ and 
measurement $\ell$, under assumptions expressing
that $m$ and $n$ are local dimensions. More precisely, assume that
the map $\vp$ defining $V_n$ is differentiable at some $x_0$ in the interior of $S$, that
$\ell$ is differentiable at $v_0=\vp(x_0)$, and that both tangent maps
have full rank at these points, that is,
\[
{\rm dim}(d\vp_{x_0}(\R^n))=n\quad{\rm and}\quad {\rm dim}(d\ell_{v_0}(V))=m.
\]
Then, by taking $v_1=v_0$ and $v_2=\vp(x_0+tx)$ in the quotient that defines $\mu_Z$,
and letting $t\to 0$ for arbitrary $x\in \R^n$, one finds that
\[
\mu_Z \geq  \max_{v\in d\vp_{x_0}(\R^n)}\frac{\|v\|_V}{\|d\ell_{v_0}(v)\|_Z},
\]
and therefore it is infinite if $n\leq m$, by the same argument as in the previous remark.
\end{remark}

\subsection{The best fit estimator}

We define a first recovery map $z\mapsto \t u=R(z)$ as the best fit estimator
in the $Z$ norm
\be
\label{best-fit-est}
\t u\coloneqq{\rm argmin}\{ \|z-\ell (v)\|_Z \;: \; v\in V_n\}.
\ee
The existence of such a minimizer is trivial if the space $V_n$ and the
measurement map $\ell$ are linear. It can also be ensured in the nonlinear case under additional
assumptions, for example compactness of the set $S$ defining the nonlinear space $V_n$,
which will be the case in the application to shape recovery discussed in \S \ref{section5}.
If the minimizer does not exist, we may consider a near minimizer, that is $\t u\in V_n$ satisfying
\[
\|z-\ell (\t u)\|_Z\leq C \|z-\ell (v)\|_Z, \quad v\in V_n,
\] 
for some fixed $C>1$. Inspection of the proofs of our main results below reveals that similar recovery
bounds can be obtained for such a near minimizer, up to the multiplicative constant $C$.

Recall that our assumption \eqref{noisemodel} on the noise model is a control on $\|\eta\|_p$ 
for some $1\leq p\leq \infty$. For this value of $p$, we introduce the quantity 
\[
\beta_Z\coloneqq \max_{z\in \R^m}\frac{\|z\|_Z}{\|z\|_p}
\]
We are now in position to state a recovery bound in this general framework.

\begin{theorem}
\label{theobestfit}
The best fit estimator $\t u$ from \eqref{best-fit-est} satisfies the estimate
\be
\|u-\t u\|_V \leq C_1e_n(u)+C_2\|\eta\|_p,
\label{bestfitbound}
\ee
where $C_1\coloneqq1+2\alpha_Z\mu_Z$ and $C_2\coloneqq2\beta_Z\mu_Z$. 
\end{theorem}

\noindent
{\bf Proof:} Consider any $v\in V_n$ and write
\[
\|u-\t u\|_V \leq \|u-v\|_V +\|v-\t u\|_V \leq \|u-v\|_V +\mu_Z\|\ell(v)-\ell(\t u)\|_Z,
\]
where we have used \iref{invstab}. On the other hand, the minimizing
property of $\t u$ ensures that 
\[
\|\ell(v)-\ell(\t u)\|_Z\leq \|z-\ell(v)\|_Z+ \|z-\ell(\t u)\|_Z\leq 2 \|z-\ell(v)\|_Z.
\]
Furthermore, using the stability \iref{stab1} of $\ell$ and the definition of $\beta_Z$, we have
\[
\|z-\ell(v)\|_Z \leq \|\ell(v)-\ell(u)\|_Z+  \|\eta\|_Z  \leq \alpha_Z \|u-v\|+\beta_Z\|\eta\|_p.
\]
Combining the three estimates, we reach
\[
 \|u-\t u\|_V\leq (1+2\alpha_Z\mu_Z) \|u-v\|_V+ 2\beta_Z\mu_Z \|\eta\|_p,
 \]
 which gives \iref{bestfitbound} by optimizing over $v\in V_n$.
 \hfill $\Box$
 \newline
 
The constants $C_1$ and $C_2$ in the above recovery estimate depend on the
choice of norm $\|\cdot\|_Z$. Note that they are invariant when this norm is scaled
by a factor $t>0$, since this has the effect of multiplying $\alpha_Z$ and $\beta_Z$
by $t$ and dividing $\mu_Z$ by $t$, which is consistant with the fact that the 
resulting estimator $\t u$ is left unchanged by such a scaling. In the next section
we show, in the particular setting of linear measurements,
that specific choices of $\|\cdot\|_Z$ can be used to minimize $C_1$ or $C_2$. 
This setting also allows us to introduce and study a generalized interpolation estimator,
which is not relevant to the present section since the nonlinear measurement map $\ell$ 
is not assumed to be surjective: in the presence of noise, there 
might exist no $v\in V$ that agrees with the data, in the sense that $z=\ell(u)+\eta$ does not
belong to the range of $\ell$.

\section{Linear observations}
\label{section3}

In this section, we assume that the $\ell_i\in V'$
are independent linear functionals, still allowing $V_n$ to be a general nonlinear space.
In this framework, which contains the example of shape recovery discussed in \S \ref{section5}, one has 
\[
\alpha_Z=\max_{v\in V}\frac{\|\ell(v)\|_Z}{\|v\|_V}
\]
and
\[
\mu_Z=\max_{v\in V_n^\diff}\frac{\|v\|_V}{\|\ell(v)\|_Z}, 
\]
where
\[
V_n^\diff=V_n-V_n\coloneqq\{v_1-v_2\;:\; v_1,v_2\in V_n\}.
\]
In this particular setting, we can identify the norms $\|\cdot\|_Z$ that minimize
the constants $C_1\coloneqq1+2\alpha_Z\mu_Z$ and $C_2\coloneqq2\beta_Z\mu_Z$, respectively.

\subsection{Optimal norms}

As $\ell:V\to \R^m$ is continuous and surjective, we can define a norm on $\R^m$ through
\be
\|z\|_W=\min\{\|v\|_V \;: \; \ell(v)=z\}.
\label{riesznorm}
\ee

\begin{remark}
If $V$ is a Hilbert space, the minimizer is unique by strict convexity of $\|\cdot\|_V$, and the $m$-dimensional space
\[
W:=\{\argmin_{\ell(v)=z}\|v\|_V, \; z\in \R^m \}
\]
is exactly the span of the Riesz representers of the observation functionals $\ell_i\in V'$. Moreover, denoting $P_W$ the orthogonal projection on $W$, we have 
\[
\|\ell(v)\|_W=\|P_Wv\|_V,\quad v\in V.
\]
For this reason, we sometimes refer to $\|\cdot\|_W$ as the {\em Riesz norm} even in the case
of a more general Banach space.
\end{remark}

The following result shows that the choice $\|\cdot\|_Z\coloneqq\|\cdot\|_W$ is the one that minimizes the
constant $C_1$, while $C_2$ is minimized by simply taking the $\ell^p$ norm $\|\cdot\|_Z=\|\cdot\|_p$.

\begin{theorem}
\label{theonorm}
For any norm $\|\cdot\|_Z$, one has
\[
\alpha_W\mu_W =\mu_W\leq \alpha_Z \mu_Z,
\]
and 
\[
\beta_p\mu_p=\mu_p \leq \beta_Z\mu_Z,
\]
where $(\alpha_W,\beta_W,\mu_W)$ and $(\alpha_p,\beta_p,\mu_p)$ are
the triplets $(\alpha_Z,\beta_Z,\mu_Z)$ when $\|\cdot\|_Z\coloneqq\|\cdot\|_W$ and $\|\cdot\|_Z=\|\cdot\|_p$,
respectively.
\end{theorem}

\noindent
{\bf Proof:} One has
\[
\alpha_W=\max_{v\in V}\frac{\|\ell(v)\|_W}{\|v\|_V}=\max_{z\in \R^m}\max_{\ell(v)=z}\frac{\|z\|_W}{\|v\|_V}=1,
\]
and so
\[
\alpha_W\mu_W=\mu_W=\max_{v\in V_n^\diff}\frac{\|v\|_V}{\|\ell(v)\|_W}
\leq 
\max_{v\in V_n^\diff}\frac{\|\ell(v)\|_Z}{\|\ell(v)\|_W}\max_{v\in V_n^\diff}\frac{\|v\|_V}{\|\ell(v)\|_Z} 
=\max_{v\in V_n^\diff}\frac{\|\ell(v)\|_Z}{\|\ell(v)\|_W}\mu_Z.
\]
We now observe that from the definition of $W$, one has 
\[
\max_{v\in V_n^\diff}\frac{\|\ell(v)\|_Z}{\|\ell(v)\|_W}\leq
\max_{z\in \R^m}\frac{\|z\|_Z}{\|z\|_W}
=\max_{z\in \R^m} \max_{\ell(v)=z}\frac{\|z\|_Z}{\|v\|_V}=\alpha_Z.
\]
We have thus obtained the first claim $\alpha_W\mu_W =\mu_W\leq \alpha_Z \mu_Z$.
For the second claim, note that we trivially have $\beta_p=1$, and so
\[
\beta_p\mu_p=\mu_p=\max_{v\in V_n^\diff}\frac{\|v\|_V}{\|\ell(v)\|_p}
\leq \max_{v\in V_n^\diff}\frac{\|\ell(v)\|_Z}{\|\ell(v)\|_p}\max_{v\in V_n^\diff}\frac{\|v\|_V}{\|\ell(v)\|_Z}
\leq \beta_Z\mu_Z.
\]
\hfill $\Box$

\begin{remark}
In the particular case where $V$ is a Hilbert space, $V_n$ a linear subspace and $p=2$, 
it was already observed in \cite{BGM2017} that the reconstruction operators
based on the choice $\|\cdot\|_Z=\|\cdot\|_W$ or $\|\cdot\|_Z=\|\cdot\|_2$
are the most stable with respect to the
approximation error and the noise error, respectively. The above result may thus be
seen as a generalization of this state of affairs to the case of nonlinear 
subspaces of Banach spaces, and $\ell^p$ noise.
\end{remark}

\subsection{The generalized interpolation estimator}

Thanks to the surjectivity of $\ell$,
we may introduce the space
\[
V_z := \{v\in V\;: \; \ell(v)=z\},
\]
and consider the minimization problem
\[
\min_{v\in V_z}\min_{\t v\in V_n} \|v-\t v\|_V.
\]
If $(u^*,\t u)\in V_z\times V_n$ is a minimizing pair, the function $u^*$ is given by
\[
u^*=u^*(z)\in {\rm argmin}\{ {\rm dist}(v,V_n)_V \;: \; \ell (v)=z\},
\]
and is called the generalized interpolation estimator, since it exactly matches the data.

\begin{remark} 
The best fit and generalized interpolation estimation
may be thought of as the two extreme cases, $t\to \infty$
and $t\to 0$, of the penalized estimator
\[
u_t\coloneqq{\rm argmin}\{\|z-\ell (v)\|_Z+t\,{\rm dist}(v,V_n)_V\}.
\]
As explained earlier, the generalized interpolation operator 
may not be well defined in the general case where 
the $\ell_i$ are nonlinear. As opposed to the best fit, 
or the above penalized estimator $u_t$ when $t>0$, 
the generalized interpolation estimator does not involve the choice of a 
particular norm $Z$. 
\end{remark}

On the other hand, we see that $\t u$ is the solution to the problem
\[
\min_{\t v\in V_n} {\rm dist}(\t v,V_z)_V.
\]
Observing that
\[
{\rm dist}(\t v,V_z)_V=\min_{\ell(v)=z}\|\t v-v\|_V=\min_{\ell(v')=\ell(\t v)-z}\|v'\|_V=\|\ell(\t v)-z\|_W,
\]
we thus find that $\t u$ is precisely the best fit estimator for the Riesz norm $\|\cdot\|_Z\coloneqq\|\cdot\|_W$.
\newline

In the Hilbert space setting, the generalized interpolation estimator $u^*$ is therefore the orthogonal 
projection of this particular best fit estimator $\t u$ onto
the affine space $V_z$. It may thus also be derived from $\t u$ by the correction
procedure
\[
u^*=\t u + w-P_W\t u,
\]
where $w=\argmin_{\ell(v)=z}\|v\|_V\in W$ is the preimage by $\ell$ of the measurements $z$.
In the noiseless case when $w=P_W u$, this correction can only improve the approximation
since it reduces the component of $u-\t u$ in the $W$ direction while leaving unchanged
the orthogonal component, and so, in view of Theorems \ref{theobestfit} and \ref{theonorm}, we are ensured that
\[
\|u-u^*\|_V \leq C_1e_n(u),
\]
where $C_1\coloneqq 1+2\mu_W$.
\newline

More generally, in the noisy case, and without the assumption that $V$ is a Hilbert space, there is no guarantee that $u^*$ performs better than $\t u$, but we still obtain
an error estimate on $u^*$ that is similar in nature to that satisfied by $\t u$.

\begin{theorem}
The generalized interpolation estimator $u^*$ satisfies the estimate
\be
\|u-u^*\|_V \leq C_1e_n(u)+C_2\|\eta\|_p,
\label{genintbound}
\ee
where $C_1\coloneqq 2+2\mu_W$ and $C_2\coloneqq (1+2\mu_W)\beta_W$. 
\end{theorem}

\noindent
{\bf Proof:} 
Take $\delta\in \argmin_{\ell(v)=\eta}\|v\|_V$, so that $\ell(\delta)=\eta$ and $\|\eta\|_W=\|\delta\|_V$.
For $v$ and $v^*$ in $V_n$, decompose
\be
\label{decomposition u-u*}
\|u-u^*\|_V\leq \|u-v\|_V +\|v-v^*\|_V +\|v^*-u^*\|_V.
\ee
For the middle term, using \iref{invstab}, we write

\begin{align*}
\|v-v^*\|_V&\leq \mu_W\|\ell(v-v^*)\|_W\\
&\leq \mu_W(\|\ell(v-u)\|_W+\|\ell(u-u^*)\|_W+\|\ell(u^*-v^*)\|_W)\\
&\leq \mu_W (\|v-u\|_V+\|\eta\|_W+\|u^*-v^*\|_V)
\end{align*}
since $\alpha_W=1$, so the decomposition \iref{decomposition u-u*} becomes
\[
\|u-u^*\|_V\leq (1+\mu_W)\|u-v\|_V +\mu_W\|\eta\|_W +(1+\mu_W)\|v^*-u^*\|_V.
\]
To bound the last term, we optimize over the choice of $v^*\in V_n$ and use the definition of $u^*$ to obtain
\[
\inf_{v^*\in V_n}\|v^*-u^*\|_V={\rm dist}(u^*,V_n)\leq {\rm dist}(u+\delta,V_n)\leq {\rm dist}(u,V_n)+\|\delta\|_V=e_n(u)+\|\eta\|_W
\]
since $\ell(u+\delta)=\ell(u)+\eta=z$. Combining the last two estimates and optimizing over $v\in V_n$ gives
\[
\|u-u^*\|_V\leq (2+2\mu_W)e_n(u) +(1+2\mu_W)\|\eta\|_W,
\]
and the result follows from the definition of $\beta_W$.
\hfill $\Box$

\section{Shape recovery from cell averages}
\label{section4}

\subsection{The shape recovery problem}

The problem of reconstructing a function $u$ from
its cell averages
\[
a_T(u)\coloneqq\frac 1 {|T|} \int_T u, \quad T\in \cT,
\]
where $\cT$ is a partition of the domain $D\subset \R^d$ in which $u$ is defined,
appears naturally in two areas:
\begin{itemize}
\item
In $2d$ or $3d$ image processing, it corresponds
to the so-called super-resolution problem, that is, reconstructing a high
resolution image from its low resolution version defined on the coarse
grid $\cT$ of pixels or voxels.
\item
In numerical simulation of hyperbolic conservation laws, it plays
a central role when developing finite volume schemes on the 
computation mesh $\cT$.
\end{itemize}

Standard reconstruction methods are challenged when the function 
$u$ exhibits jump discontinuities which are not well resolved
by the partition $\cT$. Such discontinuities correspond to edges in 
image processing or shocks in conservation laws.
Here we may focus on the very simple case of characteristic 
functions of sets 
\[
u=\Chi_\Omega,
\]
that already carry the main difficulty. Therefore we are facing a problem of 
reconstructing a shape $\Omega$ from local averages of $\Chi_\Omega$.

As a simple example we work in the domain $D=[0,1]^2$ with a uniform grid
based on square cells of sidelength $h=\frac 1 L$ for some $L>1$,
therefore of the form
\[
\cT=\cT_h\coloneqq\{T_{i,j}=[(i-1)h,ih]\times [(j-1)h,jh]\;: \; i,j=1,\dots,L\}.
\]
The cardinality of the grid is therefore 
\[
n:=\#(\cT)=L^2=h^{-2}.
\]

We consider classes of characteristic functions $\Chi_\Omega$
of sets $\Omega\subset D$ with boundary of a prescribed
H\"older smoothness. The definition
of these classes requires some precision.

\begin{definition}
For $s\geq 1$, $0<R<1/2$ and $M>0$, we define the class 
$\cF_{s,R,M}$ as consisting of all characteristic
functions $\Chi_\Omega$ of domains $\Omega\subset [R,1-R]^2 \subset D$
with the following property: for all $x\in D$ there exists an orthonormal system $(e_1,e_2)$ 
and a function $\psi\in \cC^s$ with $\|\psi\|_{\cC^s}\leq M$, such that
\[
y\in \Omega \iff z_2\leq \psi(z_1),
\]
for any $y=x+z_1e_1+z_2e_2$ with $|z_1|,|z_2|\leq R$.
\label{def F alpha}
\end{definition}

Here, we have used the usual definition 
\[
\|\psi\|_{\cC^s}=\sup_{0\leq k\leq \lfloor s\rfloor} \|\psi^{(k)}\|_{L^\infty([-R,R])}
+\sup_{s,t\in [-R,R]} |s-t|^{ \lfloor s\rfloor-s} \Big |\psi^{(\lfloor s\rfloor)}(s)-\psi^{(\lfloor s\rfloor)}(t)\Big |,
\]
for the H\"older norm. In the case of integer smoothness, we use the convention that $\cC^s$
denotes functions with Lipschitz derivatives up to order $s-1$, so that in particular
the case $s=1$ corresponds to domains with Lipschitz boundaries.

\begin{remark}
The condition $\Omega\subset [R,1-R]^2$ imposing that $\Omega$ remains away from the boundary $\partial D$ might be quite restrictive in some applications; instead, one can assume that the domains $\Omega$ and $D$ are periodic, or symmetrize $\Omega$ with respect to $\partial D$.
\end{remark}

\subsection{The failure of linear reconstruction methods}

The most trivial linear reconstruction method consists in the piecewise constant approximation
\be
\t u=\sum_{T\in \cT} a_T(u) \Chi_T.
\label{pw cst on each cell}
\ee
The approximation rate of this reconstruction over the class $\cF_{s,R,M}$
is as follows.

\begin{proposition}
Let $u=\Chi_\Omega\in \cF_{s,R,M}$, its piecewise constant approximation $\t u$ by average values on each cell, defined in \iref{pw cst on each cell}, satisfies
\[
\|\Chi_\Omega-\t u\|_{L^q}\leq Ch^{\frac 1 q}=Cn^{-\frac 1 {2q}}, 
\]
where the constant $C$ depends on $R$ and $M$.
\label{linear upper bound}
\end{proposition}

\noindent
{\bf Proof:} Let $N=\lceil (\sqrt 2 R)^{-1}\rceil$, and partition the domain $D=[0,1]^2$ into $N^2$ squares of side $1/N$. Then each subsquare $Q$ is contained in the set $\{x+z_1e_1+z_2e_2, |z_1|,|z_2|\leq R\}$ from Definition \ref{def F alpha}, where $x$ is the center of $Q$. Thus $\partial\Omega$ is the restriction of the graph of an $M$-Lipschitz function on $Q$, so its arc length is bounded by
\[
|\partial\Omega\cap Q|\leq {\rm diam}(Q)\sqrt{1+M^2}\leq 2R\sqrt{1+M^2}.
\] 
As any curve of arclength $h$ intersects at most four cells from $\cT$, $\partial\Omega\cap Q$ intersects at most $4\lceil 2R\sqrt{1+M^2}/h\rceil$ cells, and summing over all subsquares, $\partial\Omega$ intersects at most $4N^2\lceil 2R\sqrt{1+M^2}/h\rceil$ cells. Denoting $\cT_{\partial\Omega}$ the set of these cells, and observing that $u|_{T}\equiv a_T(u)\in\{0,1\}$ for $T\notin\cT_{\partial\Omega}$, we get
\[
\|\Chi_\Omega-\t u\|_{L^q}^q=\sum_{T\in \cT}\int_T|u-a_T(u)|^q\leq \sum_{T\in \cT_{\partial\Omega}}|T|= h^2|\cT_{\partial\Omega}|\leq 24\frac{\sqrt{1+M^2}}{R}h 
\]
for $h\leq R$, and this bound also holds for $h>R$ since $\|\Chi_\Omega-\t u\|_{L^q}^q\leq 1$.
\hfill $\Box$
\newline

The next result shows, for the particular case $q=2$, that no better rate can actually 
be achieved by any linear method, regardless of the smoothness $s$ of the boundary. 
We conjecture that a similar result holds for $1\leq q\leq \infty$.
This motivates the use of nonlinear recovery methods, 
which are the object of the next section. 

We recall that the Kolmogorov $n$-width of a 
compact set $S$ from some Banach space $V$ is defined by
\[
d_n(S)_V:={\rm inf}_{\dim(E)\leq n}{\rm dist}(S,E)_V,  
\]
where ${\rm dist}(S,E)_V:=\max_{u\in S}\min_{v\in E} \|u-v\|_V$ and the infimum is taken over all finite dimensional spaces $E$ of dimension at most $n$.

\begin{proposition} Let $s\geq 1$ be arbitrary. Then for $R$ sufficiently small, and $M$ sufficiently large,
there exists $c>0$ such that 
the Kolmogorov $n$-widths of the class $\cF_{s,R,M}$ satisfy
\[
d_n(\cF_{s,R,M})_{L^2} \geq cn^{-\frac 1 {4}}, \quad n\geq 1.
\]
\end{proposition}

\noindent
{\bf Proof:} The proof of this result relies on similar lower bounds
for dictionaries of $d$-dimensional ridge functions
\[
\Pp_k^d:=\{x\mapsto \sigma_k(\omega\cdot x+b) \; : \; \|\omega\|_2=1, \; c_1\leq b\leq c_2\}
\]
where $\sigma_k(t):=\max\{0,t\}^k$ is the so-called RELU-$k$ function. Here, we work in
the space $L^2(B)$ where $B$ is an arbitrary ball of $\R^d$, and the constants $(c_1,c_2)$ are taken as
the inf and sup of $\omega\cdot x$ as $x\in B$ and $\|\omega\|_2=1$, respectively, that is 
we take all $b$ such that the line discontinuity of the $k$-th derivative of $ \sigma_k(\omega\cdot x+b)$ crosses the ball $B$.
Theorem 9 from \cite{SX2021}, which improves on earlier results from \cite{Makovoz1996}, shows that if 
\[
B_1(\Pp_k^d) := \overline{ \Big\{\sum_{j=1}^n a_j g_j \; : \; n\in\N, \; g_j\in\Pp_k^d, \; \sum_{j=1}^n |a_j|\leq 1\Big\} }
\]
denotes the symmetrized convex hull of this dictionary (the closure being taken in $L^2(B)$), then
\[
d_n(B_1(\Pp_k^d) )_{L^2(B)} \geq c n^{-\frac {2k+1}{2d}}, \quad n\geq 1,
\]
where $c$ depends on $k$, $d$, and the diameter of $B$. 

In our case of interest we work with the value $d=2$ and $k=0$, so that the ridge functions are
simply the characteristic functions of half-planes. By convexity, we have
\[
d_n(\Pp_0^2) )_{L^2(B)}= d_n(B_1(\Pp_0^2) )_{L^2(B)} \geq c n^{-\frac 1 4}.
\]
We take for $B$ the ball of center $(1/2,1/2)$ and radius $1/4$, which is inside our domain
$D=[0,1]^2$. It is then readily seen that for $R$ small enough and $M$ large enough, we can
extend any ridge function $g\in \Pp_0^2$ into a characteristic function
$\Chi_\Omega$ from $\cF_{s,R,M}$, as illustrated in Figure \ref{ext}.

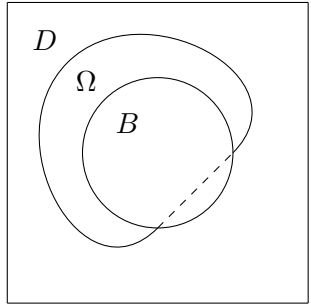
\begin{figure}[ht]
\begin{center}
\hspace{-3cm}
\begin{minipage}{.3\textwidth}
\begin{tikzpicture}[scale=1]
  \draw(-2,-2)--(-2,2);
  \draw(2,-2)--(2,2);
  \draw(-2,-2)--(2,-2);
  \draw(-2,2)--(2,2);
  \node at (-0.4,0.4){$B$};
  \node at (-0.95,0.95){$\Omega$};
  \node at (-1.5,1.5){$D$};
  \draw[dashed](0,-1)--(1,0);
  \draw(0,0) circle (1);
  \draw (0,-1) .. controls (-1,-2) and (-2.2,0.2) ..
      (-1.2,1.2) .. controls (-0.2,2.2) and (2,1) .. (1,0);
\end{tikzpicture}
\end{minipage}
\end{center}
\caption{Example of extension of the indicator of a half-plane on $B$ to the indicator of a smooth domain $\Omega$ on $D$}
\label{ext}
\end{figure}

Observing that if $E_D$ is a linear subspace of $L^2(D)$ of dimension at most $n$,
its restriction $E_B$ to $B$ is a linear subspace of $L^2(B)$ of dimension at most $n$,
and one has
\[
{\rm dist}(\Chi_\Omega,E_B)_{L^2(B)}\leq {\rm dist}(\Chi_\Omega,E_D)_{L^2(D)}.
\]
By infimizing, it follows that 
\[
d_n(\cF_{s,R,M})_{L^2(D)} \geq d_n(\Pp_0^2)_{L^2(B)} \geq c n^{-\frac 1 4},
\]
which concludes the proof. \hfill $\Box$

\begin{remark}
The fact that we impose conditions on $R$ and $M$ in the above statement is
natural since the class $\cF_{s,R,M}$ becomes empty if $R$ is not
small enough and $M$ not large enough, due to the fact that the sets $\Omega$
are assumed to be contained in the interior of $D$.
\end{remark}

\begin{remark}
The above results are easily extended to higher dimension $d\geq 2$,
with a similar definition for the class $\cF_{s,R,M}$. The rate of approximation
in $L^q$ norm by piecewise constant functions on uniform partitions is then $n^{-\frac 1 {dq}}$, 
which in the case $q=2$ is proved by a similar argument to be the best achievable by any linear reconstruction 
method. We conjecture that the same holds for more general $1\leq q\leq \infty$.
\end{remark}

\section{Shape recovery by nonlinear least-squares}
\label{section5}

\subsection{Nonlinear reconstruction on a stencil}

We now discuss a nonlinear reconstruction method for $u\in \cF_{s,R,M}$, whose output 
$\t u$ is the indicator of a domain $\t \Omega$ with polygonal boundary : on each cell $T$, the domain
$\t \Omega$ coincides with a certain half plane. In order to define the delimiting line
we only use the average values of $u$ on a $3\times 3$ stencil of cells centered at $T$.

We assume that $h<R$, so that $\Omega$ does not intersect the boundary cells $T_{i,j}$ with  $i$ or $j$ in $\{1,L\}$, and fix indices $1<i,j<L$. For the cell $T=T_{i,j}$, denote
$\o x=((i-\frac{1}{2})h,(j-\frac{1}{2})h)$ its center, and
\[
S=[(i-2)h,(i+1)h]\times [(j-2)h,(j+1)h]=\bigcup_{i-1\leq i' \leq i+1,\,j-1\leq j' \leq j+1}T_{i'j'}
\]
the stencil composed of $T$ and its $8$ neighboring cells. We define the nonlinear approximation space
\be
V_2:=\left\{\Chi_{\vec n \cdot (x-\o x)\geq c}\; : \;  \:\vec n\in\mathbb S^1, c\in \R \right\},
\label{def V_2}
\ee
which is a two-parameter family as each function is determined by $(\arg \vec n,c)\in [0,2\pi)\times \R$,
where $\arg \vec n$ is the angle of $\vec n$ with respect to the horizontal axis.

Here, our measurements are the average values of $u$ on the cells contained in $S$
\[
\ell(u)=(a_{T'}(u))_{T'\subset S}\in \R^9.
\]
In order to find a reconstruction of $u$ in $V_2$ based on these measurements, we need an inverse stability property of the form \iref{invstab}. This is not possible here, since $\ell$ cancels on all functions $\Chi_\Omega\in V_2$ with $\Omega\cap S=\emptyset$. We therefore restrict the nonlinear family $V_2$, and consider only indicators of half-planes whose boundary passes through the central cell $T$:
\be
V_{2,T}:=\Big\{\Chi_\Omega\in V_2, \partial\Omega\cap T\neq\emptyset\Big\}=\Big\{\Chi_{\vec n \cdot (x-\o x)\geq c}, \:\vec n\in\mathbb S^1, |c|\leq \frac{h}{2}|\vec n|_1\Big\}.
\label{def VT}
\ee
In this setting, we prove the existence of the following stability constants for 
$V=L^1(S)$ and $Z=\ell^1$,
which is the best norm on $\R^m$ in view of Theorem \ref{theonorm}.
For notational simplicity, we omit the reference to $Z$ in these constants.

\begin{proposition}
\label{prop stencil}
One has 
\begin{equation}
\|\ell(u)\|_{1}\leq \alpha \|u\|_{L^1(S)}, \quad u\in L^1(D),
\label{alpha stencil}
\end{equation}
and
\begin{equation}
\|u-v\|_{L^1(S)}\leq \mu \|\ell(u-v)\|_{1},\quad u,v\in V_{2,T},
\label{mu stencil}
\end{equation}
where $\alpha= h^{-2}$ and $\mu=\frac{3}{2}h^{2}$ are the optimal constants.
\end{proposition}

The proof of the stability property \iref{alpha stencil} is trivial since
on each cell 
\[
|a_{T'}(u)|\leq |T'|^{-1}\|u\|_{L^1(T')}=h^{-2}\|u\|_{L^1(T')},
\] 
with equality in case $u$ does not change sign.
The proof of the inverse stability \iref{mu stencil} is quite technical and left to the appendix.

Given the noisy observation
\[
z=\ell(u)+\eta \in \R^9,
\]
we define the estimator of $u$ on the cell $T$ by
\be
\t u_T\in \argmin_{v\in V_{2}}\|z-\ell(v)\|_1.
\label{locrec}
\ee
Here we minimize over all $V_2$, that is on all indicators of half planes, but we note that 
we may restrict to half-planes whose boundary passes through the stencil $S$.

The following result, which uses Proposition \ref{prop stencil},
shows that its distance to $u$ in $L^1(T)$ 
is comparable to the error between $u$ and its
best approximation in the $L^1(S)$ norm
\[
\o u_S:=\argmin_{v\in V_{2}}\|u-v\|_{L^1(S)}.
\]
\begin{lemma}
For all $u\in \cF_{s,R,M}$, one has
\[
\|u-\t u_T\|_{L^1(T)}\leq C_1\|u-\o u_S\|_{L^1(S)}+2\beta\mu \|\eta\|_p,
\]
where $C_1=1+2\alpha\mu=4$ and $C_2=2\beta\mu=3^{3-\frac 2 p}h^2$,
with $\alpha,\mu$ as in Proposition \ref{prop stencil}, and $\beta=9^{1-\frac 1 p}$ the
maximal ratio between $\ell^p$ and $\ell^1$ norm in $\R^9$.
\end{lemma}

\begin{proof}  We distinguish two cases:
\begin{itemize}
\item
If $\t u_T\in V_{2, T}$ and $\o u_S\in V_{2,T}$, that is, both boundaries pass through the central cell $T$,
we apply Theorem \ref{bestfitbound}
 together with Proposition \ref{prop stencil}

\begin{align*}
\|u-\t u_T\|_{L^1(T)}\leq \|u-\t u_T\|_{L^1(S)}
&\leq C_1\min_{v\in V_{2,T}}\|u-v\|_{L^1(S)}+C_2\|\eta\|_{p}\\
&= C_1\|u-\o u_S\|_{L^1(S)}+C_2\|\eta\|_{p}.
\end{align*}
with $C_1=1+2\alpha\mu$, $C_2=2\beta\mu$.
\item
Otherwise, either $\t u_T$ or $\o u_S$ has constant value $0$ or $1$ on $T$, so $\t u_T-\o u_S$  
has constant sign on $T$, and thus

\begin{align*}
\|\o u_S-\t u_T\|_{L^1(T)}&=h^2|a_{T}(\t u_T-\o u_S)|\leq h^2\|\ell(\t u_T-\o u_S)\|_1\leq h^2(\|\ell(\o u_S)-z\|_1+\|\ell(\t u_T)-z\|_1)\\
&\leq 2h^2\|\ell(\o u_S)-z\|_1
\leq 2h^2\|\ell(\o u_S-u)\|_1+2h^2\|\eta\|_1
\leq 2\|u-\o u_S\|_{L^1(S)}+2h^2\beta\|\eta\|_p.
\end{align*}
\end{itemize}
By triangle inequality, it follows that
\[
\|u-\t u_T\|_{L^1(T)}\leq 3\|u-\o u_S\|_{L^1(S)}+2h^2\beta\|\eta\|_p,
\]
which has better constants than in the estimate obtained in the first case, since the constant $C_0$ is larger than $1$.
\end{proof}

The order of the best local approximation error $\|u-\o u_S\|_{L^1(S)}$ that appears
as a bound for the reconstruction error $\|u-\t u_T\|_{L^1(T)}$ depends on the smoothness 
of the boundary, as expressed in the following lemma.

\begin{lemma}
For all $u\in \cF_{s,R,M}$, with $R\geq\frac{3}{\sqrt 2}h$, one has
\[
\|u-\o u_S\|_{L^1(S)}\leq M(3\sqrt 2 h)^{\min(s,2)+1}.
\] 
\end{lemma}

\begin{proof} We apply the definition of $\cF_{s,R,M}$ at point $\o x$: as $R\geq\frac{3}{\sqrt 2}h$, the stencil $S$ is contained in the domain
\[
\{\o x+z_1e_1+z_2e_2, \;|z_1|,|z_2|\leq R\},
\]
so $u|_S$ is the indicator of a domain delimited by a $\cC^s$ function $\psi$, with $\|\psi\|_{\cC^s}\leq M$. From the
definition of $\cC^s$, there exists an affine function $\xi$ such that
\[
|\psi(z_1)-\xi(z_1)|\leq M(3\sqrt 2 h)^{\min(s,2)},\quad |z_1|\leq \frac{3}{\sqrt 2}h.
\]
Then the function $v:\o x+z_1e_1+z_2e_2 \mapsto \Chi_{z_2\leq \xi(z_1)}$ belongs to $V_2$, and we have
\[
\|u-\o u_S\|_{L^1(S)}\leq \|u-v\|_{L^1(S)}\leq M(3\sqrt 2 h)^{\min(s,2)+1}.
\]
\end{proof}

\subsection{Global nonlinear reconstruction}

We now consider the process of recovering $u\in  \cF_{s,R,M}$ globally from its
data 
\[
z=\ell(u)+\eta,
\] 
where now $\ell(u):=(a_T(u))_{T\in \cT}\in \R^n$ and $\eta\in \R^n$ is the noise vector.
Applying to each inner cell $T\in \cT$ the previous reconstruction procedure based on the $3\times 3$ stencil $S$
centered at $T$, we obtain a global recovery $\t u=\t u(z)$ such that
\[
\t u|_{T}=\t u_T|_{T},\quad T=T_{i,j}\in \cT,\quad 1<i,j<L,
\]
where $\t u_T$ is the local estimator from \iref{locrec}. On the boundary cells $T=T_{i,j}$ with $i$ or $j$ in $\{1,L\}$, $u|_T$ is zero by Definition $\ref{def F alpha}$ so we simply set $\t u|_T=0$.
Note that $\t u$ is of the form
\[
\t u =\Chi_{\t \Omega},
\]
where $\t \Omega$ has piecewise linear boundary with respect to the mesh $\cT$. The following result gives a
global approximation bound, which confirms the improvement over linear methods when $s>1$.

\begin{theorem}
For all $u\in \cF_{s,R,M}$, one has 
\[
\|u-\t u\|_{L^q(D)}\leq C_1n^{-\frac{\min(1,s/2)}{q}}+C_2n^{-\frac{1}{pq}}\|\eta\|_p^{\frac{1}{q}}.
\]
\end{theorem}

\begin{proof}
First notice that if the result is proved for $p=q=1$,
as $u-v$ has values in $\{-1,0,1\}$,
\[
\|u-v\|_{L^q(D)}^q=\|u-v\|_{L^1(D)}\leq C_1n^{-1}+C_2n^{-1}\|\eta\|_1\leq \(C_1^{\frac{1}{q}}n^{-\frac{1}{q}}+C_2^{\frac1q}n^{-\frac{1}{pq}}\|\eta\|_p^{\frac{1}{q}}\)^q,
\]
so it suffices treat the case $p=q=1$. 

By an argument similar to the proof of Proposition \ref{linear upper bound}, $\partial\Omega$ intersects at most $16N^2\lceil 2R\sqrt{1+M^2}/h\rceil$ stencils of $9$ cells. Using the fact that $u=\o u_S$ is a constant on any other stencil, we get

\begin{align*}
\|u-\t u\|_{L^1(D)}&=\sum_{T\text{ inner cell}}\|u-\t u\|_{L^1(T)}\leq \sum_{T\text{ inner cell}}(1+2\alpha\mu)\|u-\o u\|_{L^1(S)}+2\beta\mu\|\eta\|_{\ell^1(S)}\\
&\leq 16N^2\left\lceil \frac{2R\sqrt{1+M^2}}h\right\rceil M(3\sqrt 2 h)^{\min(s,2)+1}+18\beta\mu\|\eta\|_1\leq C_1h^{\min(s,2)}+C_2 h^2\|\eta\|_1.
\end{align*}
We conclude by recalling that $n=h^{-2}$.
\end{proof}

\begin{remark}
Here the convergence rate for the noiseless term $n^{-\frac{\min(1,s/2)}{q}}$ is limited due to the use of polygonal domains in the reconstruction. So the best approximation rate $h^{\frac 2 q}=n^{-\frac 1 q}$ is already attained for 
$\cC^2$ boundaries. When the smoothness parameter $s$ is larger than $2$, better rates $n^{-\frac{s}{2q}}$ should be reachable if we use non-linear approximation spaces that are richer than the space $V_2$, for example
indicator functions of domains with boundary that have a higher order polynomial description rather than
straight lines. Of course, the stable identification of these approximants in the sense of \iref{invstab}
might require stencils that are of larger size than $3\times 3$.
\end{remark}

\begin{remark}
If $\|\eta\|_\infty\leq \frac{1}{9}$, 
then $\t u$ is exactly equal to $u$ on any cell whose corresponding stencil does not intersect $\partial\Omega$, so the error is concentrated on $\mathcal O(\sqrt n)$ cells, leading to an improved rate $n^{-\frac{p+1}{2pq}}$ instead of $n^{-\frac{1}{pq}}$ for the noise term.
\end{remark}

\subsection{Numerical illustration}

We study the behavior of the above discussed linear and non-linear 
recovery methods from cell averages for the particular target function 
$u=\chi_{\Omega}$, with $\Omega$a slightly decentered disk of radius $r=0.325$.

The linear method consists of the piecewise constant approximation \iref{pw cst on each cell},
referred to as {\it PiecewiseConstant}.
As to the nonlinear method, for the local best fit problem, we use the $\ell^2$ norm 
on $\R^9$ instead of the $\ell^1$ norm. By norm equivalence on $\R^9$, the same convergence
results can be proved to hold with different constants. This method, which we refer to
as {\it LinearInterface}, does not ensure consistency of the reconstruction in the sense that
$a_T(\t u)=a_T(u)$. One way to approach this consistency property is to modify the $\ell^2$ norm
by putting a large weight on the central cell. We refer to this variant as {\it LinearInterfaceCC},
here taking the weight $100$.

\begin{figure}[H]
\centerline{\includegraphics[width=0.7\textwidth]{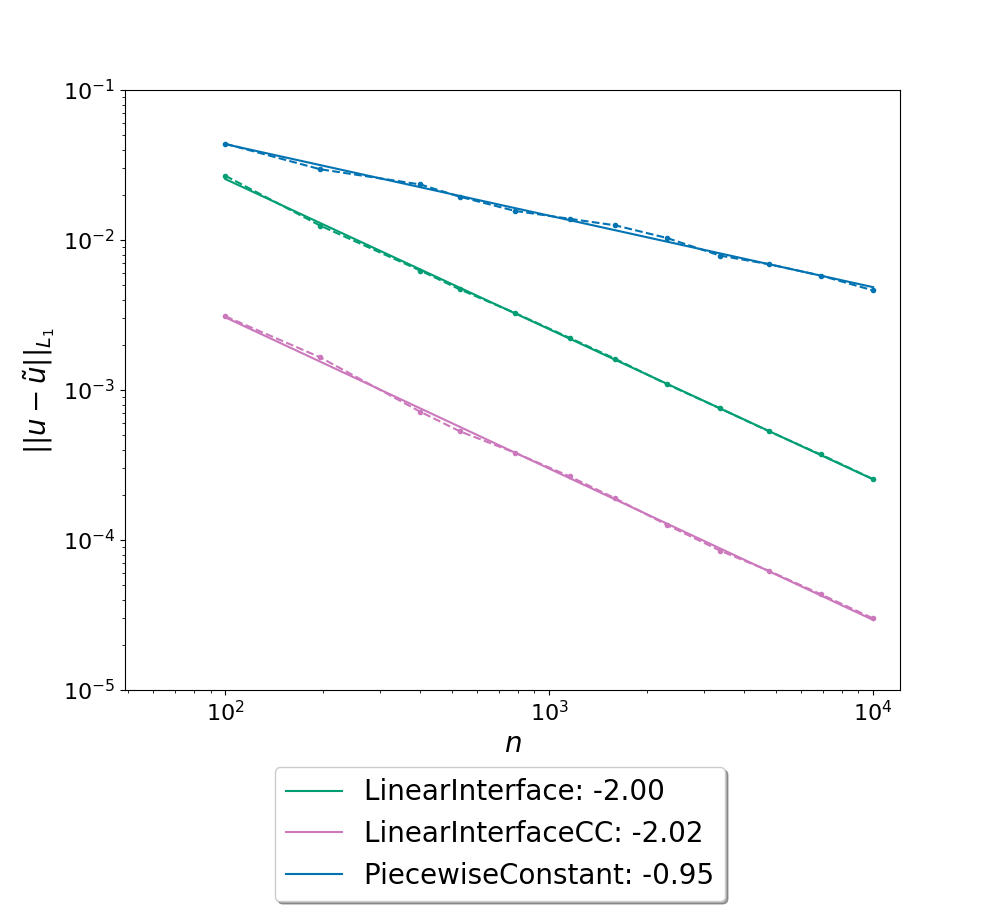}}
\caption{Convergence curves for the linear and nonlinear recovery methods}
\label{fig:plot_convergence}
\end{figure}

\begin{figure}
\centering
\begin{subfigure}{0.45\textwidth}
    \includegraphics[width=0.95\textwidth]{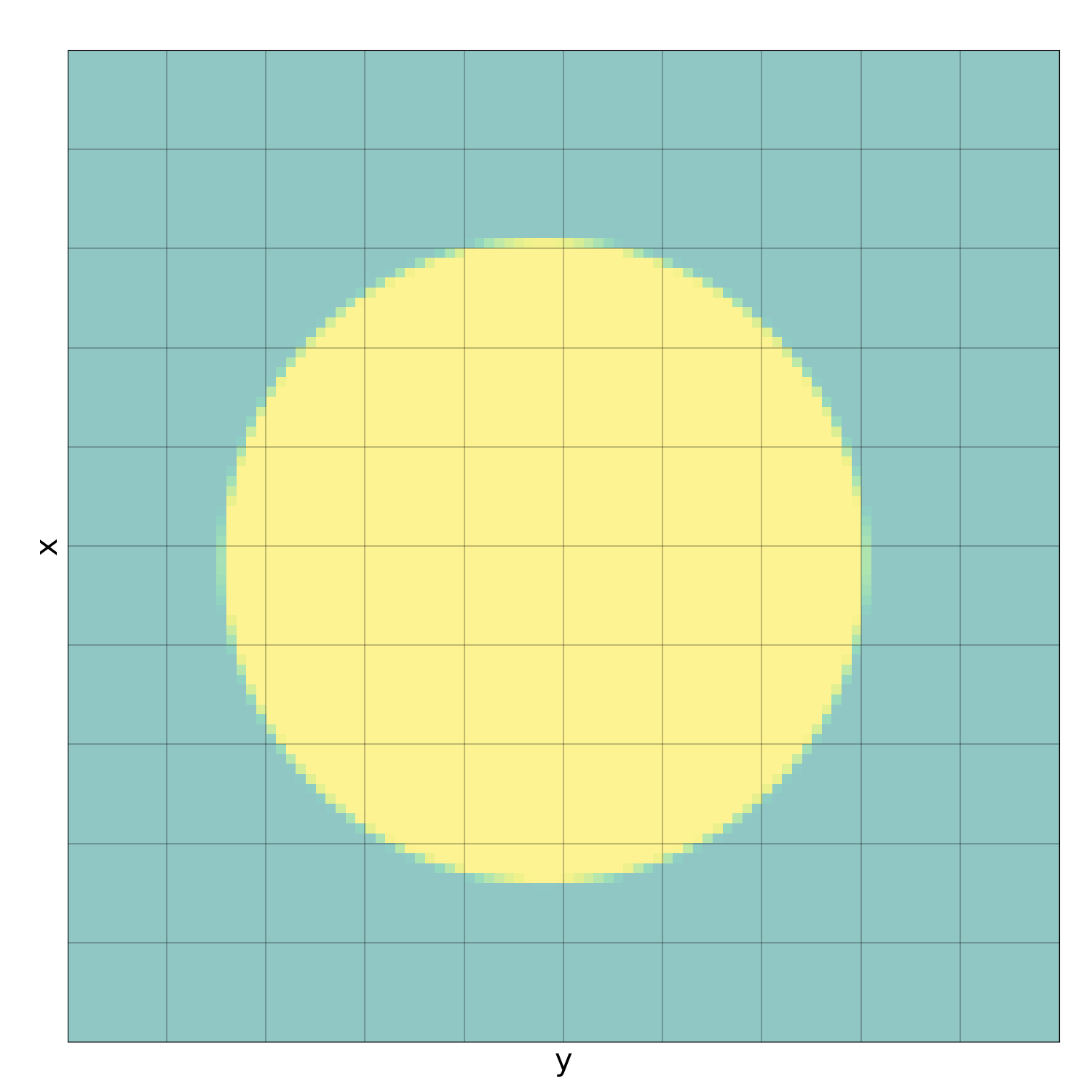}
    \caption{}
    \label{fig:NonLinearReconstruction:first}
\end{subfigure}
\hfill
\begin{subfigure}{0.45\textwidth}
    \includegraphics[width=0.95\textwidth]{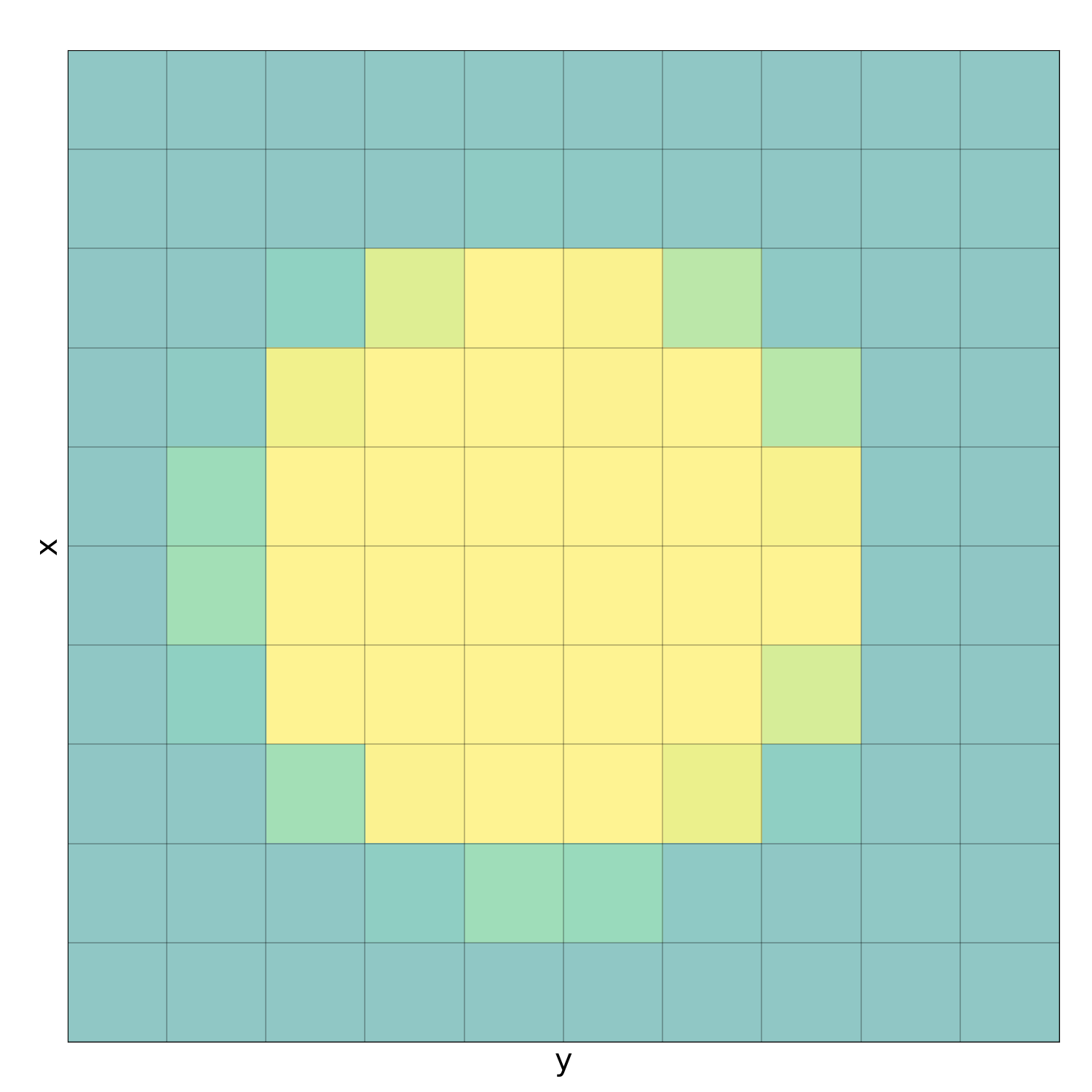}
    \caption{}
    \label{fig:NonLinearReconstruction:second}
\end{subfigure}
\hfill
\begin{subfigure}{0.45\textwidth}
    \includegraphics[width=0.95\textwidth]{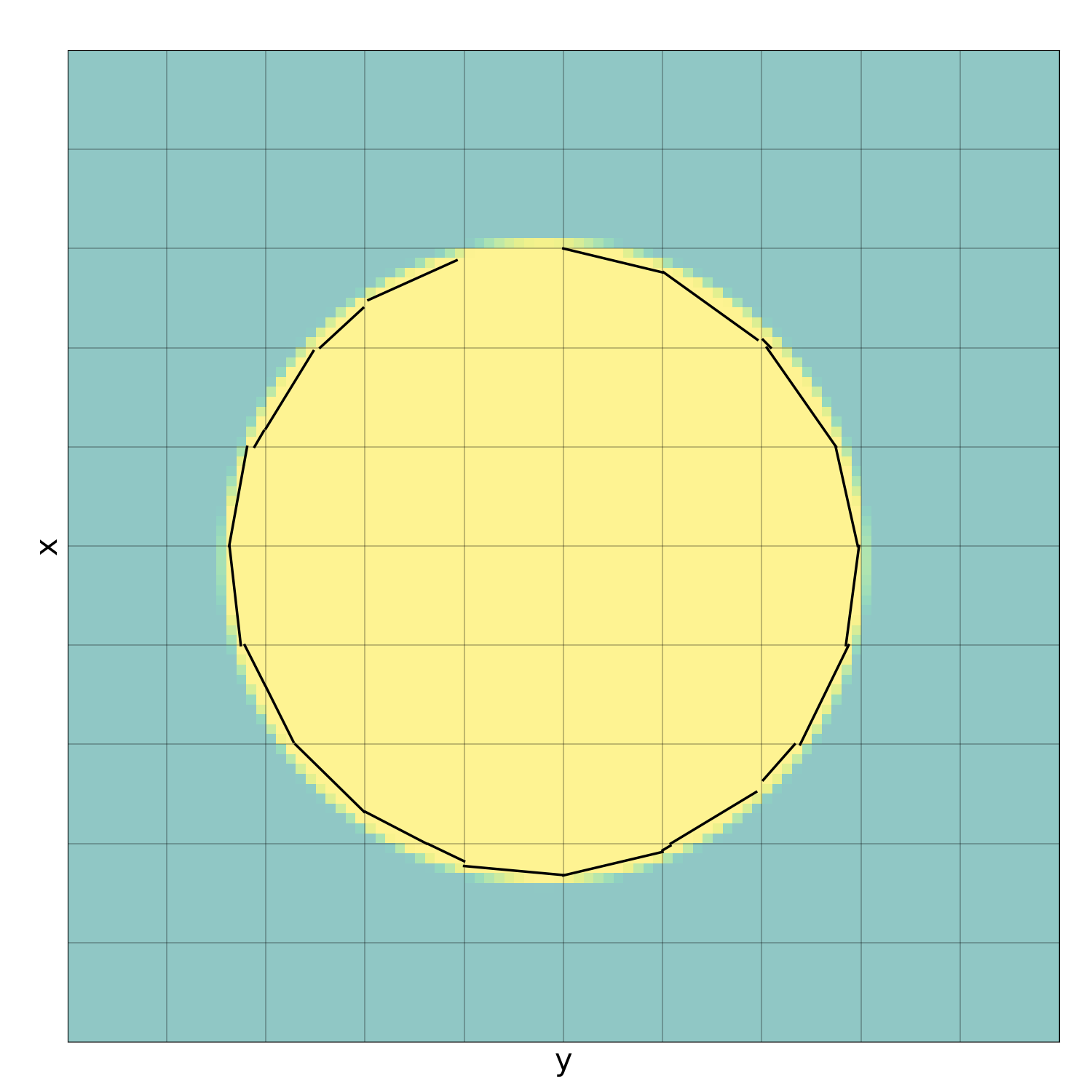}
    \caption{}
    \label{fig:NonLinearReconstruction:third}
\end{subfigure}
\hfill
\begin{subfigure}{0.45\textwidth}
    \includegraphics[width=0.95\textwidth]{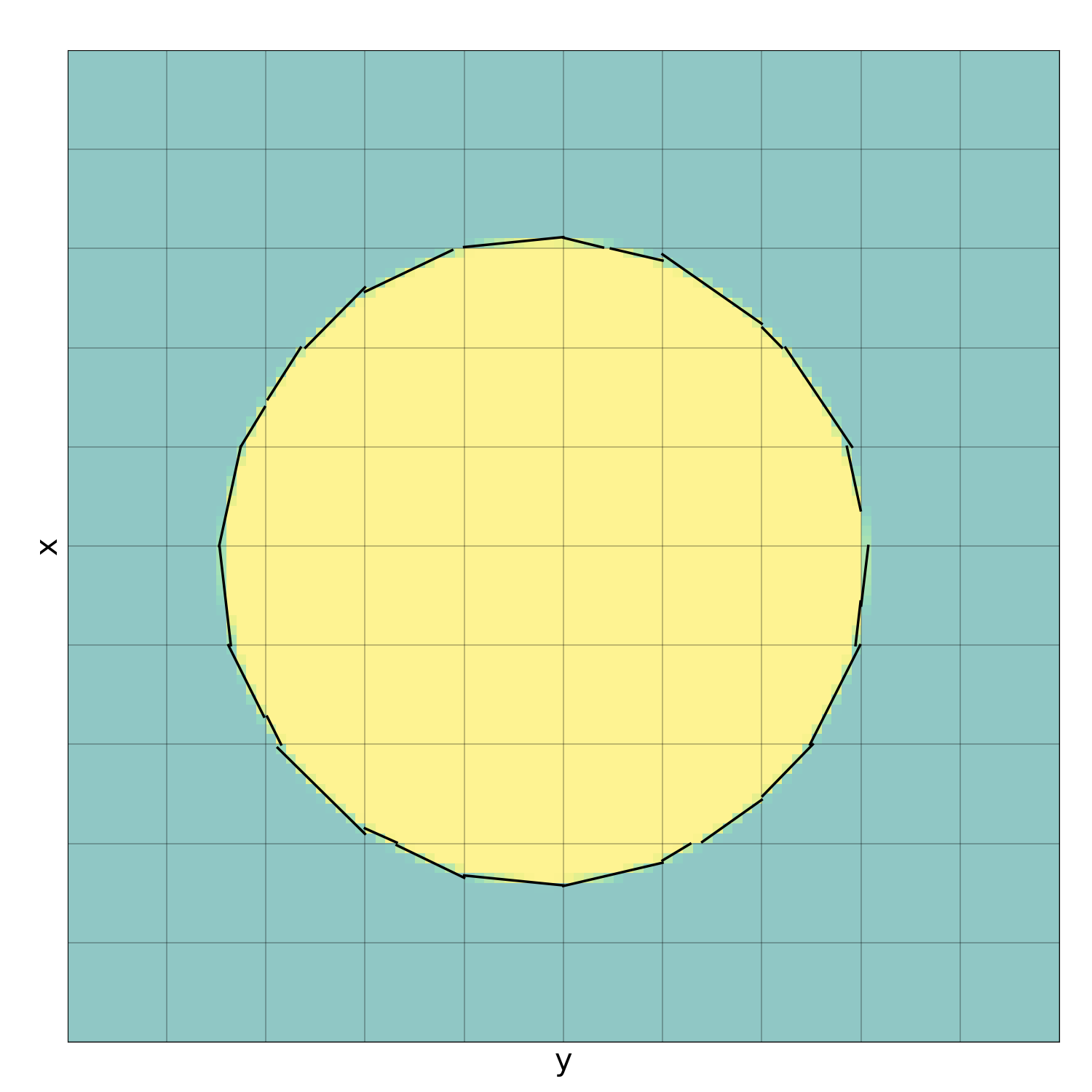}
    \caption{}
    \label{fig:NonLinearReconstruction:fourth}
\end{subfigure}        
\caption{(a) The target function, (b) its recovery by PiecewiseConstant showing the cell-average data,
and the recovered boundaries by (c) LinearInterface
and (d) LinearInterfaceCC methods.}
\label{fig:NonLinearReconstruction}
\end{figure}

Figure \ref{fig:plot_convergence} shows the convergence rates of the three methods in the $L^1$ norm. The expected $h^{-2}$ decay is observed in both non-linear methods while the linear method lays behind with a decay rate of $h^{-1}$. It is relevant to note that although both non-linear methods benefit from the same rate, the associated constants differ by an order of magnitude, showing the practical improvement gained by imposing consistency. 
This improvement is also visible on Figure \ref{fig:NonLinearReconstruction} 
which shows that in the LinearInterface method, the interfaces that minimize the $l_2$ error on the $9$ surrounding cells lay always inside the circle as the curvature of the boundary pushes them towards the center. On the contrary, LinearInterfaceCC seems to find the right compromise between sticking to the cell average while capturing at the same time the curvature trend hinted by the surrounding cell averages.

\section{Relation to compressed sensing} 
\label{section6}

\subsection{Compressed sensing and best $n$-term approximation}

In this section we discuss the application of our setting to
the sparse recovery of large vectors from a few linear observations. We thus take
\[
V=\R^N,
\]
equipped with some given norm $\|\cdot\|_V$ of interest. The linear measurements of $u=(u_1,\dots,u_N)^\top\in \R^N$ are given by
\[
(\ell_1(u),\dots,\ell_m(u))^\top=\Phi u,
\]
where $\Phi$ is an $m\times N$ measurement matrix, with typically $m\ll N$.

The topic of compressed sensing deals with sparse 
recovery of $u$ from such measurements, that is, searching to recover
an accurate approximation to $u$ by a vector with only a few non-zero components.
We refer to \cite{CTR} for some first highly celebrated breakthrough results 
and to \cite{FR} for a general treatment. 

We define the nonlinear space of $n$-sparse vectors as
\[
V_n:=\Big\{u\in \R^N \; : \; \|u\|_0:=\#\{i\; : \; u_i\neq 0\}\leq n\Big\},
\]
and the best $n$-term approximation error in the $V$ norm as
\[
e_n(u)_V:=\min_{v\in V_n} \|u-v\|_V.
\]
One natural question is to understand for which 
type of measurement matrices $\Phi$ does the
noise-free measurement $y=\Phi u$ contain enough information, in order to recover
any $u$ up to an error $e_n(u)_V$. In other words, one asks if there exists a recovery 
map $R: \R^m \to \R^N$ such that one has the {\it instance optimality property}
at order $n$
\be
\|u-R(\Phi u)\|_V\leq C_0 e_n(u)_V, \quad u\in \R^N,
\label{instanceopt}
\ee
with $C_0$ a fixed constant, which we denote by $IOP(n,C_0)$. This question has been answered in \cite{CDD2009} 
in terms of the null space $\cN:=\{v\in \R^N \; :\; \Phi v=0\}$. We say that $\Phi$ satisfies the {\it null space property} 
at order $k$ with constant $C_1$, denoted by $NSP(k,C_1)$ if and only if
\be
\|v\|_V \leq C_1e_k(v)_V, \quad v\in \cN.
\label{nsp}
\ee
This property quantifies how much vectors from the null space can be concentrated on
a few coordinates. One main result of \cite{CDD2009} is the equivalence between $IOP$ at order
$n$ and $NSP$ at order $2n$ in the following sense.

\begin{theorem}
One has $IOP(n,C_0)\Rightarrow NSP(2n,C_0)$ and conversely
$NSP(2n,C_1)\Rightarrow IOP(n,2C_1)$.
\end{theorem}

One natural question is whether matrices $\Phi$ with such properties can be constructed
with a number of rows/measurements $m$ barely larger than $n$. As we recall
further the answer to this question is strongly tied to the norm $V$ used on $\R^N$.

\subsection{Stability and the null space property}

The nonlinear estimation results that we have obtained in \S \ref{section2} and \S \ref{section3} can be applied to
the setting of sparse recovery, offering us a different vehicle than the null space property 
to establish instance optimality. 

In the present setting, for a given norm $\|\cdot\|_Z$, the stability property \iref{stab1} takes the form
\be
\|\Phi u\|_Z\leq \alpha_Z \|u\|_V, \quad u\in \R^N
\ee
and the inverse stability property \iref{invstab} takes the form
\be
\|v\|_V \leq \mu_Z\|\Phi v\|_Z, \quad v\in V_{2n},
\ee
since for sparse vectors we have $V_n^{\rm diff}=V_n-V_n=V_{2n}$. We refer to these properties
as $S(\alpha_Z)$ and $IS(2n,\mu_Z)$, respectively.

Application of Theorem \ref{theobestfit} in the noiseless case
immediately gives us that the nonlinear best fit recovery $R(\Phi u)=\t u$ satisfies
the instance optimality bound \iref{instanceopt} with constant $C_0=1+2\alpha_Z\mu_Z$.
In other words
\be
S(\alpha_Z)\; {\rm and} \; IS(2n,\mu_Z) \Rightarrow IOP(n,C_0), \quad \quad C_0=1+2\alpha_Z\mu_Z.
\label{SISIOP}
\ee
The following result shows that $(S,IS)$ is actually equivalent to $NSP$, and
thus to $IOS$, in the sense that a converse result holds when $\|\cdot\|_Z$ is chosen to be
the Riesz norm \iref{riesznorm}.

\begin{theorem}
For any norm $\|\cdot\|_Z$, one has 
\be
S(\alpha_Z)\; {\rm and} \; IS(2n,\mu_Z) \Rightarrow NSP(2n,C_1), \quad \quad C_1=1+\alpha_Z\mu_Z.
\label{SISNSP}
\ee
Conversely, let $\|\cdot\|_W$ be the Riesz norm so that $\|\Phi u\|_W=\min_{\Phi v=\Phi u}\|v\|_V$, 
then
\be
NSP(2n,C_1) \Rightarrow S(\alpha_W)\; {\rm and} \; IS(2n,\mu_W) , \quad \quad \alpha_W=1\; {\rm and}\;
\mu_W=1+C_1.
\ee
\end{theorem}

\begin{proof}
Assume that $S(\alpha_Z)$ and $IS(2n,\mu_Z)$ hold. Let $v\in\cN$ and $\tilde v$ its best approximation in $V_{2n}$, then

\begin{align*}
\|v\|_V&\leq \|v-\tilde v\|_V+ \|\tilde v\|_V\\
&\leq e_{2n}(v)_V+ \mu_Z\|\Phi \tilde v\|_W\\
&=e_{2n}(v)_V+ \mu_Z\|\Phi (v-\tilde v)\|_W \leq (1+\alpha_Z\mu_Z)e_{2n}(x)_V.
\end{align*}
This shows that $NSP(2n,C_1)$ holds with $C_1=1+\alpha_Z\mu_Z$.

Conversely, assume that $NSP(2n,C_1)$ holds. From the definition of the Riesz norm, it is
immediate that $S(\alpha_W)$ holds with $\alpha_W=1$. For $v\in V_{2n}$, let $\t v$ be the 
minimizer of $\min_{\Phi \t v=\Phi v}\|\t v\|_V$. Then, one has
\[
\|v\|_V\leq \|\t v\|_V+\|v-\t v\|_V\leq \|\t v\|_V+C_1\sigma_{2n}(v-\t v)_V\leq (1+C_1)\|\t v\|_V,
\]
by using $v$ as a sparse approximation to $v-\t v$. Since $\|\t v\|_V=\|\Phi v\|_W$ , this shows that $S(2n,\mu_W)$ holds with $\mu_W=1+C_1$.
\end{proof}

\subsection{The case of $\ell^p$ norms}

The range of $m$ allowing the properties to be fulfilled is best understood in the case of 
the $\ell^p$ norms, that is $\|\cdot \|_V=\|\cdot\|_p$, as discussed in \cite{CDD2009} which
points out a striking difference between the case $p=2$ and $p=1$:
\begin{enumerate}
\item
In the case $p=2$, it is proved that $NSP(2,C_1)$ cannot hold unless $N\leq C_1^2m$.
In other words, instance optimality in $\ell^2$ even at order $n=1$ requires a number
of measurements that is proportional to the full space dimension. 
\item
In the more favorable case $p=1$, it is proved that for matrices 
which satisfy the $\ell^2$-RIP property of order $3n$ 
$$
(1-\delta) \|v\|_2^2 \leq \|\Phi v\|_2^2 \leq (1+\delta) \|v\|^2_2, \quad v\in V_{3n},
$$ 
with parameter $0<\delta<\frac{(\sqrt 2-1)^2}3$, the $NSP(2n,C_1)$ holds with $C_1$ depending on $\delta$. Such matrices are known to exists with $m\sim n\log(N/n)$ rows. 
\end{enumerate}

Our setting based on the stability properties $S$ and $IS$ applies more naturally 
to a different class of matrices built from graphs, which is also known to be well adapted
for sparse recovery in the $\ell^1$ norm. A bipartite graph with $(N,m)$ left and 
right vertices, and of  left degree $d$, is an $(l,\eps)$-graph expander if
\[
 |X|\leq l \Rightarrow |N(X)| \geq d(1-\eps)|X|,\quad   X\subset \{1,\dots,N\},
\]
where $N(X)\subset \{1,\dots,m\}$ is the set of vertices connected to $X$. We necessarily have 
$|N(X)|\leq d|X|$, and $(1-\eps)dl\geq m$. From \cite{CRVW}, it is known that  there exists a $(2n, \frac{1}{2})$-graph expander with $d\sim \log \frac{N}{n}$ and $m\sim nd\sim n\log(N/n)$.

Now denote $\Phi\in \{0,1\}^{m\times N}$ the adjacency matrix of this graph, so that each column of $\Phi$ has $d$ nonzero entries. Then
\[
\|\Phi x\|_1\leq d\|x\|_1,\quad x\in \R^N,
\]
and
\[
\|\Phi x\|_1\geq d(1-\eps)\|x\|_1,\quad x\in V_{2n}.
\]
Therefore $S(\alpha_1)$ and $IS(2n,\mu_1)$, hold with $\alpha_1=d$ and $\mu_1=\frac{1}{d(1-\eps)}=\frac 2 d$, which 
by \iref{SISNSP} and \iref{SISIOP} gives $NSP(2n,C_1)$ with $C_1=3$ and $IOP(n,C_0)$ with $C_0=5$.

\section{Appendix: Proof of Proposition \ref{prop stencil}} 
\begin{figure}[ht]
\begin{center}
\includegraphics[scale=0.5]{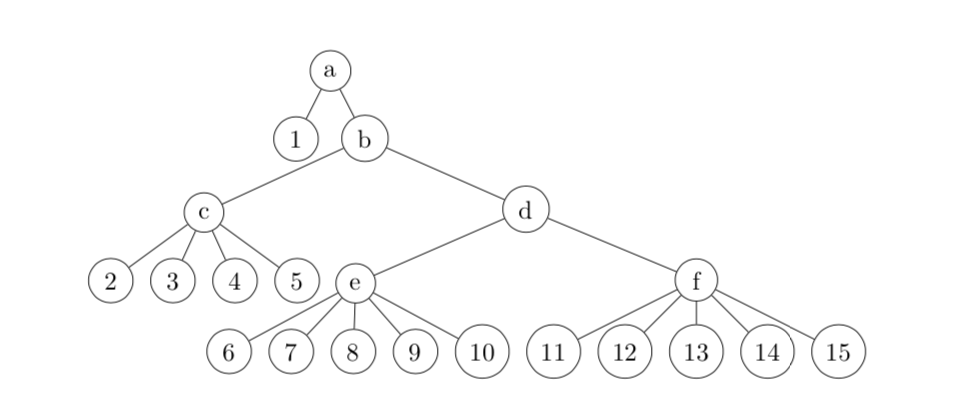}
\caption{Structure of the proof, each leaf corresponds to a different case, and each node contains a general treatment valid for all its sons}
\label{tree}
\end{center}
\end{figure}

The proof contains 15 cases, represented on a tree in Figure \ref{tree}. These cases
correspond to different geometric situations, up to certain symmetries that leave the final relevant quantities $\|\ell(w)\|_1$
and $\|w\|_{L^1(S)}$ unchanged.
\nl
\nl
{\bf Node a:} Take $w=u-v\in V_{2,T}^{\rm diff}$, with $u,v\in V_{2,T}$,
and denote $\vec n_u$, $\vec n_v$ and $c_u$, $c_v$ the corresponding unit vectors and offsets from the definition \ref{def VT}
of $V_{2,T}$. Recalling that $\o x=(\o x_1,\o x_2)$ is the center of $S$, we also denote
\[
\Delta_u=\{x\in \R^2, (x-\o x)\cdot \vec n_u=c_u\}
\]
the delimiting line between $\{u=0\}$ and $\{u=1\}$, and define $\Delta_v$ in a similar way.
\nl
\nl
{\bf Case 1:} If $\vec n_u=\vec n_v=\vec n$, we have
\[
w=
\begin{cases}
\Chi_{c_u\leq \vec n\cdot (x-\o x)< c_v}&\text{ if }c_u\leq c_v\\
-\Chi_{c_v\leq \vec n\cdot (x-\o x)< c_u}&\text{ otherwise}
\end{cases}
\]
so $w$ has constant sign, which implies $\|w\|_{L^1(S)}=h^{2}\|\ell(w)\|_1$.
\nl

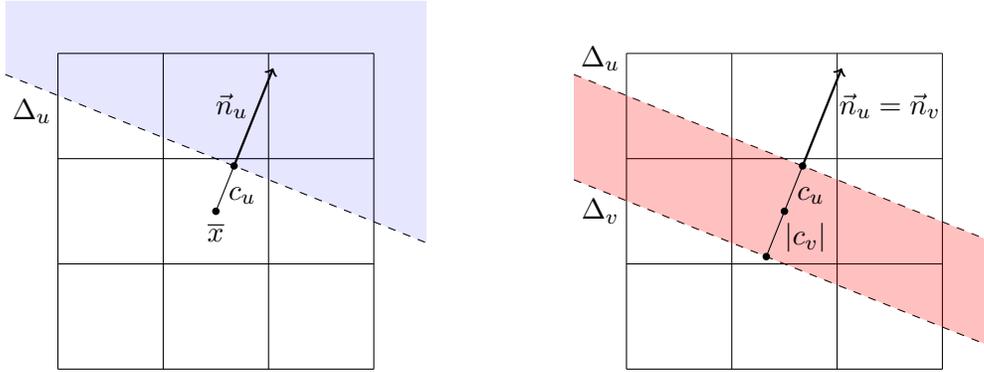
\begin{figure}[ht]
\begin{center}
\hspace{-3.5cm}
\begin{minipage}{.2\textwidth}
\begin{tikzpicture}[scale=1.4]
  \draw(-1.5,-1.5)--(-1.5,1.5);
  \draw(-0.5,-1.5)--(-0.5,1.5);
  \draw(0.5,-1.5)--(0.5,1.5);
  \draw(1.5,-1.5)--(1.5,1.5);
  \draw(-1.5,-1.5)--(1.5,-1.5);
  \draw(-1.5,-0.5)--(1.5,-0.5);
  \draw(-1.5,0.5)--(1.5,0.5);
  \draw(-1.5,1.5)--(1.5,1.5);
  \draw[dashed](-2,1.3)--(2,-0.3);
  \draw[draw=none, fill=blue, fill opacity=0.1](-2,1.3)--(2,-0.3)--(2,2)--(-2,2)--cycle;
  \draw(0,0)--(5/29,25/58);
  \node at (0.25,0.15){$c_u$};
  \draw[->, line width=0.3mm](5/29,25/58)--(5/29+4/10.77,25/58+10/10.77);
  \node at (0.15,1){${ \vec n_u}$};
  \node at (0,0) [circle,fill,inner sep=1pt]{};
  \node at (0.5/2.9,0.5/2.9*2.5) [circle,fill,inner sep=1pt]{};
  \node at (0,-0.2){$\o x$};
  \node at (-1.75,0.95){$\Delta_u$};
\end{tikzpicture}
\end{minipage}
\hspace{4cm}
\begin{minipage}{.2\textwidth}
\begin{tikzpicture}[scale=1.4]
  \draw(-1.5,-1.5)--(-1.5,1.5);
  \draw(-0.5,-1.5)--(-0.5,1.5);
  \draw(0.5,-1.5)--(0.5,1.5);
  \draw(1.5,-1.5)--(1.5,1.5);
  \draw(-1.5,-1.5)--(1.5,-1.5);
  \draw(-1.5,-0.5)--(1.5,-0.5);
  \draw(-1.5,0.5)--(1.5,0.5);
  \draw(-1.5,1.5)--(1.5,1.5);
  \draw[dashed](-2,1.3)--(2,-0.3);
   \draw[dashed](-2,0.3)--(2,-1.3);
  \draw[draw=none, fill=white, fill opacity=0.1](-2,1.3)--(2,-0.3)--(2,2)--(-2,2)--cycle;
  \draw[draw=none, fill=red, fill opacity=0.25](-2,1.3)--(2,-0.3)--(2,-1.3)--(-2,0.3)--cycle;
  \draw(-5/29,-25/58)--(5/29,25/58);
  \node at (0.25,0.15){$c_u$};
  \node at (0.2,-0.25){$|c_v|$};
  \draw[->, line width=0.3mm](5/29,25/58)--(5/29+4/10.77,25/58+10/10.77);
  \node at (1,1){${ \vec n_u=\vec n_v}$};
  \node at (0,0) [circle,fill,inner sep=1pt]{};
  \node at (0.5/2.9,0.5/2.9*2.5) [circle,fill,inner sep=1pt]{};
  \node at (-0.5/2.9,-0.5/2.9*2.5) [circle,fill,inner sep=1pt]{};
  \node at (-1.75,1.45){$\Delta_u$};
  \node at (-1.75,0){$\Delta_v$};
\end{tikzpicture}
\end{minipage}
\end{center}
\caption{Left: $3\times 3$ stencil $S$, with $\o x$ its center, and an example of function $u\in V_{2,T}$ with directing vector $\vec n_u$ and offset $c_u>0$.
Here the dotted line corresponds to $\Delta_u$, and the shaded region to $u=1$, while $u=0$ elsewhere.
Right: Representation of Case 1 ($\vec n_u=\vec n_v$), here $c_v <0 <c_u$ so $w=-1$ on the shaded region and $w=0$ elsewhere}
\end{figure}

\noindent
{\bf Node b:} In all other cases, the cones
\[
\cC_+=\{x\in \R^2,\, w(x)=1\}\quad \text{and}\quad \cC_-=\{x\in \R^2,\, w(x)=-1\}
\]
are non-empty, and we can define the external bisector
\[
\Delta=\{x\in \R^2,\,(\vec n_u-\vec n_v)\cdot(x-\o x)=c_u-c_v\},
\]
which is the line of symmetry between $\cC_+$ and $\cC_-$. We also denote
\[
\cC=\cC_+\cup \cC_-=\{x\in \R^2, |w(x)|=1\}.
\]
Observing that
\be
\|w\|_{L^1(S)}=|S\cap \cC|
\ee
and
\be
\|\ell(w)\|_1=h^{-2}{\sum_{T\subset S}}\Big| |T\cap\cC_+|-|T\cap \cC_-|\Big|,
\ee
the stability property \iref{mu stencil}
can be rewritten as
\[
|S\cap \cC| \leq \frac{3}{2}\sum_{T\subset S}\Big| |T\cap\cC_+|-|T\cap \cC_-|\Big|=\frac{3}{2}\left(|S\cap \cC| - 2\sum_{T\subset S}\min(|T\cap\cC_+|,|T\cap\cC_-|) \right),
\]
or equivalently
\begin{equation}
|S\cap \cC|\geq 6\sum_{T\subset S}\min(|T\cap\cC_+|,|T\cap\cC_-|).
\label{geometric condition}
\end{equation}
Up to a rotation of $S$ by a multiple of $\frac{\pi}{2}$, we may
assume without loss of generality that
\[
\arg(\vec n_u-\vec n_v)\in\left[\frac{\pi}4,\frac{3\pi}4\right],
\]
that is, $\Delta$ is at an angle of at most $\frac{\pi}{4}$ with the horizontal axis, and $\cC_+$ lies above $\Delta$.
Take $(\vec e_1,\vec e_2)$ the canonical basis of $\R^2$.
\nl
\nl
{\bf Node c:} Consider the situation where $(\vec n_u\cdot \vec e_2)(\vec n_v\cdot \vec e_2)>0$. As $\vec n_u\neq \vec n_v$ and $\vec n_u\neq -\vec n_v$, the lines $\Delta_u$ and $\Delta_v$ intersect at one point $X\in \R^2$.
Moreover, the above condition implies $X+\vec e_2\notin \cC$. Using the fact that $|\arg(\Delta)|\leq\frac{\pi}{4}$, we also get $X+\vec e_1\notin \cC$.

Up to a symmetry with respect to the vertical axis, we can assume that $\cC_+$ is included in the quadrant $X+\R_+^2$. Now consider a cell $T\subset S$ such that $\min(|T\cap \cC_+|,|T\cap \cC_-|)\neq0$, then there exist points $x\in T\cap \cC_-$ and $y\in T\cap \cC_+$. As $x_1 \leq X_1 \leq y_1$ and $x_2\leq X_2\leq y_2$, we get $X\in T$, so there is at most one such cell $T$, and
inequality \eqref{geometric condition} reduces to
\[
|S\cap \cC|\geq 6\min(|T\cap \cC_+|,|T\cap \cC_-|).
\]

\begin{figure}[ht]
\begin{center}
\hspace{-3.5cm}
\begin{minipage}{.2\textwidth}
\begin{tikzpicture}[scale=1.4]
  \draw(-1.5,-1.5)--(-1.5,1.5);
  \draw(-0.5,-1.5)--(-0.5,1.5);
  \draw(0.5,-1.5)--(0.5,1.5);
  \draw(1.5,-1.5)--(1.5,1.5);
  \draw(-1.5,-1.5)--(1.5,-1.5);
  \draw(-1.5,-0.5)--(1.5,-0.5);
  \draw(-1.5,0.5)--(1.5,0.5);
  \draw(-1.5,1.5)--(1.5,1.5);
  \draw[dashed](-1.8,-2)--(2,1.04);
  \draw[dashed](-1.65,-2)--(0.35,2);
  \draw[draw=none](-2,-2)--(-2,2);
  \draw[draw=none, fill=blue, fill opacity=0.1](-1.55,-1.8)--(2,1.04)--(2,2)--(0.35,2)--cycle;
  \draw[draw=none, fill=red, fill opacity=0.25](-1.55,-1.8)--(-1.8,-2)--(-1.65,-2)--cycle;
  \node at (-1.55,-1.8) [circle,fill,inner sep=1pt]{};
  \node at (-1.75,-1.7){$X$};
  \node at (1.3,1.75){$\cC_+$};
  \node at (-0.1,1.75){$\Delta_u$};
  \node at (1.75,0.5){$\Delta_v$};
\end{tikzpicture}
\end{minipage}
\hspace{4cm}
\begin{minipage}{.2\textwidth}
\begin{tikzpicture}[scale=1.4]
  \draw(-1.5,-1.5)--(-1.5,1.5);
  \draw(-0.5,-1.5)--(-0.5,1.5);
  \draw(0.5,-1.5)--(0.5,1.5);
  \draw(1.5,-1.5)--(1.5,1.5);
  \draw(-1.5,-1.5)--(1.5,-1.5);
  \draw(-1.5,-0.5)--(1.5,-0.5);
  \draw(-1.5,0.5)--(1.5,0.5);
  \draw(-1.5,1.5)--(1.5,1.5);
  \draw[dashed](-2,-0.9)--(2,0.7);
  \draw[dashed](-0.38,-2)--(0.42,2);
  \draw[draw=none, fill=blue, fill opacity=0.1](0,-0.1)--(2,0.7)--(2,2)--(0.42,2)--cycle;
  \draw[draw=none, fill=red, fill opacity=0.25](0,-0.1)--(-2,-0.9)--(-2,-2)--(-0.38,-2)--cycle;
  \node at (0,-0.1) [circle,fill,inner sep=1pt]{};
  \node at (0.2,-0.3){$X$};
  \node at (-1.1,-1.75){$\cC_-$};
  \node at (1.4,1.75){$\cC_+$};
  \draw[ line width=0.3mm](-1,-0.9)--(1,-0.9);
  \draw[ line width=0.3mm](-1,-0.9)--(-1,1.1);
  \draw[ line width=0.3mm](-1,1.1)--(1,1.1);
  \draw[ line width=0.3mm](1,-0.9)--(1,1.1);
  \draw[ line width=0.3mm](0.5,-0.5)--(0.5,0.5);
  \draw[ line width=0.3mm](-0.5,0.5)--(0.5,0.5);
  \draw[ line width=0.3mm](-0.5,-0.5)--(-0.5,0.5);
  \draw[ line width=0.3mm](-0.5,-0.5)--(0.5,-0.5);
\end{tikzpicture}
\end{minipage}

\hspace{-3.5cm}
\begin{minipage}{.2\textwidth}
\begin{tikzpicture}[scale=1.4]
  \draw(-1.5,-1.5)--(-1.5,1.5);
  \draw(-0.5,-1.5)--(-0.5,1.5);
  \draw(0.5,-1.5)--(0.5,1.5);
  \draw(1.5,-1.5)--(1.5,1.5);
  \draw(-1.5,-1.5)--(1.5,-1.5);
  \draw(-1.5,-0.5)--(1.5,-0.5);
  \draw(-1.5,0.5)--(1.5,0.5);
  \draw(-1.5,1.5)--(1.5,1.5);
  \draw[dashed](-2,-1.75)--(2,-1.75+20/9);
   \draw[dashed](-1.4,-2)--(0.2,2);
  \draw[draw=none, fill=blue, fill opacity=0.1](-1.1,-1.25)--(2,-1.75+20/9)--(2,2)--(0.2,2)--cycle;
  \draw[draw=none, fill=red, fill opacity=0.25](-1.1,-1.25)--(-2,-1.75)--(-2,-2)--(-1.4,-2)--cycle;
  \node at (-1.1,-1.25) [circle,fill,inner sep=1pt]{};
  \node at (-1.3,-1.1){$X$};
  \node at (-1.55,-1.75){$\cC_-$};
  \node at (1.3,1.75){$\cC_+$};
  \draw[ line width=0.3mm](-1.1,-1.25)--(0.7,-1.25);
  \draw[ line width=0.3mm](-1.1,-1.25)--(-1.1,1);
  \draw[ line width=0.3mm](0.7,-1.25)--(0.7,1);
  \draw[ line width=0.3mm](-1.1,1)--(0.7,1);
  \draw[ line width=0.3mm](-0.5,-1.25)--(-0.5,-0.5);
  \draw[ line width=0.3mm](-1.1,-0.5)--(-0.5,-0.5);
\end{tikzpicture}
\end{minipage}
\hspace{4cm}
\begin{minipage}{.2\textwidth}
\begin{tikzpicture}[scale=1.4]
  \draw(-1.5,-1.5)--(-1.5,1.5);
  \draw(-0.5,-1.5)--(-0.5,1.5);
  \draw(0.5,-1.5)--(0.5,1.5);
  \draw(1.5,-1.5)--(1.5,1.5);
  \draw(-1.5,-1.5)--(1.5,-1.5);
  \draw(-1.5,-0.5)--(1.5,-0.5);
  \draw(-1.5,0.5)--(1.5,0.5);
  \draw(-1.5,1.5)--(1.5,1.5);
  \draw[dashed](-2,-1.7)--(2, 0.3);
  \draw[dashed](-0.5,-2)--(0.78,2);
  \draw[draw=none, fill=blue, fill opacity=0.1](-0.1,-0.75)--(2, 0.3)--(2,2)--(0.78,2)--cycle;
  \draw[draw=none, fill=red, fill opacity=0.25](-0.1,-0.75)--(-2,-1.7)--(-2,-2)--(-0.5,-2)--cycle;
  \node at (-0.1,-0.75) [circle,fill,inner sep=1pt]{};
  \draw[ line width=0.3mm](-0.02,-0.5)--(-0.02+2/3,0.5);
  \draw[ line width=0.3mm](0.4,-0.5)--(0.4+2/3,0.5);
  \draw[ line width=0.3mm](-0.02,-0.5)--(0.4,-0.5);
  \draw[ line width=0.3mm](-0.02+2/3,0.5)--(0.4+2/3,0.5);
  \draw(-0.1,-0.75)--(-0.1,-0.5);
  \node at (-0.2,-0.625){$z$};
  \node at (0.3,-0.3){$l$};
  \node at (0.05,-1.2){$T$};
  \node at (-0.05,0.1){$\o T$};
  \draw[dotted](-0.1,-0.75)--(-0.1+35/24,-2);
  \draw[dotted](-0.3,-0.75+1.2/7)--(-2,-0.75+1.9*6/7);
  \node at (-1.75,0.5){$\Delta$};
\end{tikzpicture}
\end{minipage}
\end{center}
\caption{Cases $2$, $3$, $4$, and $5$}
\end{figure}
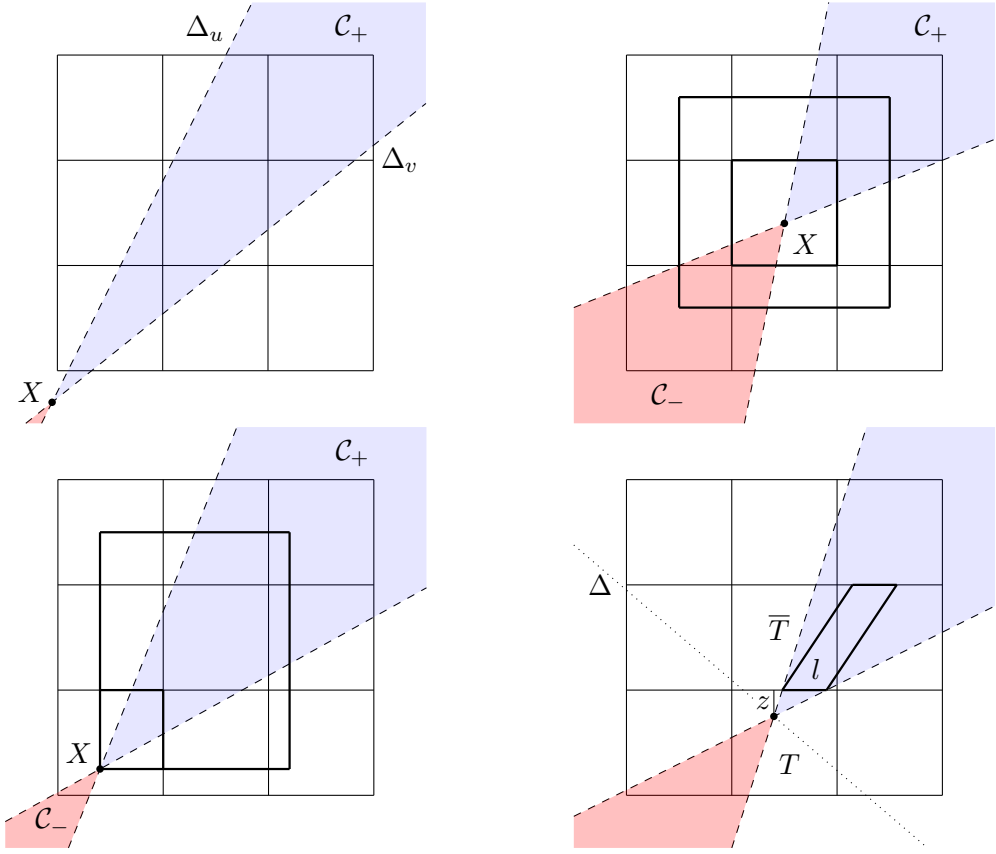

\noindent
{\bf Case 2:} If $X\notin S$, then $w$ has constant sign on $S$, so $\|w\|_{L^1(S)}=h^{2}\|\ell(w)\|_1$.
\newline
\newline
{\bf Case 3:} If $X$ is in the central cell $T$, the dilation of $T$ with respect to $X$ by a factor $2$ is a subset of $S$, and the image of $\cC\cap T$ is in $\cC\cap S$, so
\[
|S\cap \cC|\geq 4 |T\cap \cC|\geq 8\min(|T\cap \cC_+|,|T\cap \cC_-|).
\]
\noindent
{\bf Case 4:} If $X$ is in the lower left cell $T$, the dilation of $T\cap \cC_+$ with respect to $X$ by a factor $3$  is in $S\cap \cC_+$, so
\[
 |S\cap \cC|\geq  |S\cap \cC_+|\geq 9|T\cap \cC_+|\geq 9\min(|T\cap \cC_+|,|T\cap \cC_-|).
\]
The same argument holds with $\cC_-$ instead of $\cC_+$ when $X$ is in the upper right cell. Moreover, as $\Delta_u$ and $\Delta_v$ go through the central cell, $X$ may not be in the upper left or lower right cells.
\nl
\nl
{\bf Case 5:} If $X$ is in the lower central cell $T$, denote $l=|\partial T\cap \cC_+|\in (0,h)$ the distance between $\Delta_u$ and $\Delta_v$ when they pass from $T$ to the central cell $\o T$, and $z={\rm dist}(X, \o T)\in(0,h)$ the depth of the point of intersection. Then
\[
|T\cap \cC_+|=\frac{zl}{2}\quad \text{and}\quad |T\cap \cC_-|\leq \frac{zl}{2}\left(\frac{h-z}{z}\right)^2,
\]
so $\min(|T\cap \cC_+|,  |T\cap \cC_-|)\leq \frac{hl}{4}$. On the other hand, the parallelogram of base $\partial T\cap \cC_+$, of height $h$, and with sides orthogonal to $\Delta$ belongs to $(S\setminus T)\cap \cC_+$ (it does not escape to the right of $S$ because $\Delta$ is close to the horizontal axis, so the sides of the parallelogram are at an angle at most $\frac{\pi}{4}$ with the vertical axis), and has an area $hl$, which proves that
\[
|\cC\cap S|\geq hl+|\cC_+\cap T|+|\cC_-\cap T|\geq 6\min(|T\cap \cC_+|,  |T\cap \cC_-|).
\]
A similar construction can be applied to the remaining cases where $X$ is in the upper central, central left or central right cell, which concludes the proof for Node c.
\nl
\nl
{\bf Node d:} If now $(\vec n_u\cdot \vec e_2)(\vec n_v\cdot \vec e_2)\leq 0$, as $\arg(\vec n_u-\vec n_v)\in\left[\frac{\pi}4,\frac{3\pi}4\right]$, we get $\vec n_u\cdot \vec e_2\geq 0 \geq \vec n_v\cdot \vec e_2$. Observe that $\cC_++\vec e_2\subset \cC_+$ since for all $x\in \cC_+$,
\[
(x+\vec e_2-\o x)\cdot \vec n_u \geq (x-\o x)\cdot \vec n_u \geq c_u\quad \text{and}\quad (x+\vec e_2-\o x)\cdot \vec n_v \leq (x-\o x)\cdot \vec n_v < c_v.
\]
In the same way, $\cC_--\vec e_2\subset \cC_-$. We now divide $S$ into columns separated by the vertical boundaries between cells, and 
in addition by vertical lines where $\Delta$ intersects the two horizontal lines separating cells of $S$, as illustrated in Figure \ref{cones contain the vertical axis}.

\begin{figure}[ht]
\begin{center}
\hspace{-2cm}
\begin{minipage}{.2\textwidth}
\begin{tikzpicture}[scale=1.4]
  \draw(-1.5,-1.5)--(-1.5,1.5);
  \draw(-0.5,-1.5)--(-0.5,1.5);
  \draw(0.5,-1.5)--(0.5,1.5);
  \draw(1.5,-1.5)--(1.5,1.5);
  \draw(-1.5,-1.5)--(1.5,-1.5);
  \draw(-1.5,-0.5)--(1.5,-0.5);
  \draw(-1.5,0.5)--(1.5,0.5);
  \draw(-1.5,1.5)--(1.5,1.5);
  \draw[dashed](-2,1.3)--(2,-0.3);
  \draw[dashed](0.1,-2)--(0.9,2);
  \draw[draw=none, fill=blue, fill opacity=0.1](-2,1.3)--(30/54, 15/54)--(0.9,2)--(-2,2)--cycle;
  \draw[draw=none, fill=red, fill opacity=0.25](2,-2)--(0.1,-2)--(30/54, 15/54)--(2,-0.3)--cycle;
  \node at (30/54, 15/54) [circle,fill,inner sep=1pt]{};
  \draw[->, line width=0.3mm](30/54, 15/54)--(30/54+0.4*0.92848, 15/54+0.92848);
  \draw[->, line width=0.3mm](30/54, 15/54)--(30/54+0.98058, 15/54-0.2*0.98058);
  \draw[->, line width=0.3mm](30/54, 15/54)--(30/54+0.4*0.92848-0.98058, 15/54+0.92848+0.2*0.98058);
  \node at (-1.75,-0.6){$\Delta$};
  \node at (-0.25,1.75){$\cC_+$};
  \node at (1.1,-1.75){$\cC_-$};
  \node at (1.75,0.2){$\vec n_v$};
  \node at (1.2,1.2){$\vec n_u$};
  \node at (-0.45,1.1){$\vec n_u-\vec n_v$};
  \draw[dotted](-2,-1.106562)--(2,1.060232);
  \end{tikzpicture}
\end{minipage}
\hspace{4cm}
\begin{minipage}{.2\textwidth}
\begin{tikzpicture}[scale=1.4]
  \draw[line width=0.4mm](-1.5,-1.5)--(-1.5,1.5);
  \draw[line width=0.4mm](-0.5,-1.5)--(-0.5,1.5);
  \draw[line width=0.4mm](0.5,-1.5)--(0.5,1.5);
  \draw[line width=0.4mm](1.5,-1.5)--(1.5,1.5);
  \draw(-1.5,-1.5)--(1.5,-1.5);
  \draw(-1.5,-0.5)--(1.5,-0.5);
  \draw(-1.5,0.5)--(1.5,0.5);
  \draw(-1.5,1.5)--(1.5,1.5);
  \draw[dashed](-2,1.3)--(2,-0.3);
  \draw[dashed](0.1,-2)--(0.9,2);
  \draw[draw=none, fill=blue, fill opacity=0.1](-2,1.3)--(30/54, 15/54)--(0.9,2)--(-2,2)--cycle;
  \draw[draw=none, fill=red, fill opacity=0.25](2,-2)--(0.1,-2)--(30/54, 15/54)--(2,-0.3)--cycle;
  \draw[line width=0.4mm](-2+1.11974,-1.5)--(-2+1.11974,1.5);
  \draw[line width=0.4mm](0.965786,-1.5)--(0.965786,1.5);
  \draw[dotted](-2,-1.106562)--(2,1.060232);
  \draw(-1.5, -1.106562+2.1667/8)--(-1.2189,-1.0188);
  \draw(-1.2189,-1.0188)--(-0.6726,-0.1812);
   \draw(-0.6726,-0.1812)--(-0.5, -1.106562+2.1667*3/8);
   \draw(0.5, -1.106562+2.1667*5/8)--(0.7112,0.1101);
   \draw(0.7112,0.1101)--(1.25747,0.9477);
   \draw(1.25747,0.9477)--(1.5, -1.106562+2.1667*7/8);
  \end{tikzpicture}\end{minipage}
\end{center}
\caption{Generic situation for Node d, and partition of $S$ into $5$ columns: here, in addition to the $4$ vertical lines delimiting the cells of $S$, we added $2$ vertical lines passing through the intersections of $\Delta$ with the $2$ horizontal cell delimiters}
\label{cones contain the vertical axis}
\end{figure}
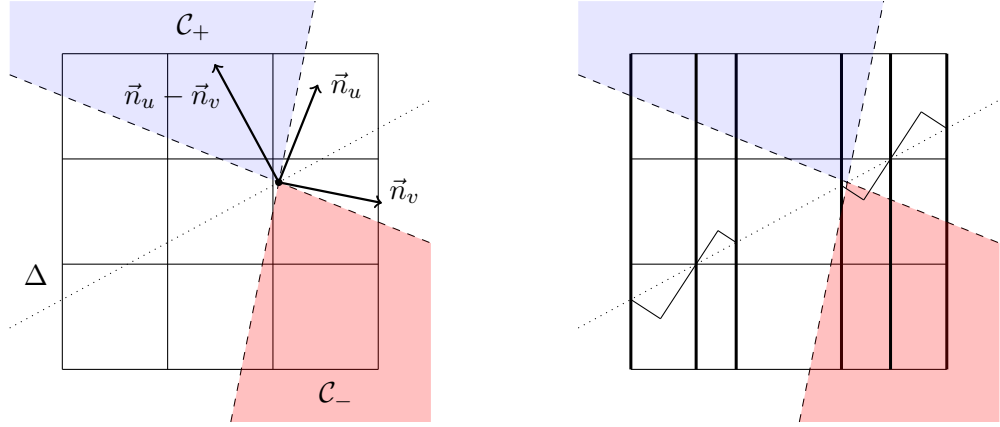

Let $U$ be such a column, and $T$ a cell intersecting $U$. If $T\cap U\neq T$, $\Delta$ intersects either the upper or lower boundary of $T$, but not both since $\Delta$ is at an angle of at most $\frac{\pi}{4}$ with the horizontal axis. If it is the upper boundary, the symmetric of the part of $T\cap U$ above $\Delta$ with respect to $\Delta$ is in $T\cap U$. If it is the lower boundary, the symmetric of the part of $T\cap U$ below $\Delta$ with respect to $\Delta$ is in $T\cap U$. Using the fact that 
$\cC_+$ and $\cC_-$ are symmetric with respect to $\Delta$, we obtain 
\[
\min(|T\cap\cC_+|,|T\cap\cC_-|)=\min(|T\cap U\cap \cC_+|,|T\cap U\cap\cC_-|)+\min(|T\cap U^c\cap \cC_+|,|T\cap U^c\cap\cC_-|).
\]
Thanks to this observation, instead of \eqref{geometric condition} we only have to prove the inequality
\begin{equation}
|U\cap \cC|\geq 6\sum_{T\subset U}\min(|T\cap U\cap \cC_+|,|T\cap U\cap \cC_-|)
\label{column}
\end{equation}
on each column $U$ separately. We thus consider only one column $U$ in the sequel, and assume up to a horizontal dilation (which preserves the condition $|\arg(\Delta)|\leq\frac \pi 4$) that $U$ has width $h$ and is composed of three full cells.

According to the definition of the columns, there is at most one cell $T\subset U$ such that $T\cap \Delta\neq\emptyset$, and as $\Delta$ separates $\cC_+$ and $\cC_-$, it is only for this cell that we may have $\min(|T\cap\cC_+|,|T\cap \cC_-|)\neq0$. If there is no such cell, \eqref{column} trivially holds.
Otherwise, similar to Node c, we only need to prove
\[
|U\cap \cC|\geq 6\min(|T\cap \cC_+|,|T\cap \cC_-|),
\]
where $T\subset U$ is the cell containing $\Delta\cap U$. Denoting $P_1$, $P_2$, $P_3$ and $P_4$ the upper left, upper right, lower left and lower right corner points of $T$, we observe that the assumptions on $\Delta$ and $U$ imply $P_1,P_2\notin \mathring\cC_-$ and $P_3,P_4\notin \mathring\cC_+$.
\nl
\nl
{\bf Node e:} If $U\cap \Delta_u\cap \Delta_v=\emptyset$, that is, if $U$ contains no intersection point between $\Delta_u$ and $\Delta_v$, we match $5$ cases depending on the position of $T$ in $U$, and of its corners with respect to $\cC$. They are illustrated in Figure \ref{node e}.

\begin{figure}[ht]
\begin{center}
\begin{minipage}{.2\textwidth}
\begin{tikzpicture}[scale=1.4]
  \draw(-0.5,-1.5)--(-0.5,1.5);
  \draw(0.5,-1.5)--(0.5,1.5);
  \draw(-0.5,-1.5)--(0.5,-1.5);
  \draw(-0.5,-0.5)--(0.5,-0.5);
  \draw(-0.5,0.5)--(0.5,0.5);
  \draw(-0.5,1.5)--(0.5,1.5);
  \draw[dashed](-0.75,-0.9)--(0.75,-0.5);
  \draw[dashed](-0.75,-1.1)--(0.75,-0.9);
  \draw[draw=none, fill=blue, fill opacity=0.1](-0.75,-0.9)--(0.75,-0.5)--(0.75,1.75)--(-0.75,1.75)--cycle;
  \draw[draw=none, fill=red, fill opacity=0.25](-0.75,-1.1)--(0.75,-0.9)--(0.75,-1.75)--(-0.75,-1.75)--cycle;
  \node at (-0.5,-0.5) [circle,fill,inner sep=1pt]{};
  \node at (-0.3,-0.3){$P_1$};
  \node at (0.5,-0.5) [circle,fill,inner sep=1pt]{};
  \node at (0.3,-0.3){$P_2$};
  \node at (-0.5,-1.5) [circle,fill,inner sep=1pt]{};
  \node at (-0.3,-1.3){$P_3$};
  \node at (0.5,-1.5) [circle,fill,inner sep=1pt]{};
  \node at (0.3,-1.3){$P_4$};
  \end{tikzpicture}
\end{minipage}
\hspace{-5mm}
\begin{minipage}{.2\textwidth}
\begin{tikzpicture}[scale=1.4]
  \draw(-0.5,-1.5)--(-0.5,1.5);
  \draw(0.5,-1.5)--(0.5,1.5);
  \draw(-0.5,-1.5)--(0.5,-1.5);
  \draw(-0.5,-0.5)--(0.5,-0.5);
  \draw(-0.5,0.5)--(0.5,0.5);
  \draw(-0.5,1.5)--(0.5,1.5);
  \draw[dashed](-0.75,-1.3)--(0.75,-0.1);
  \draw[dashed](-0.75,-1.4)--(0.75,-1.3);
  \draw[draw=none, fill=blue, fill opacity=0.1](-0.75,-1.3)--(0.75,-0.1)--(0.75,1.75)--(-0.75,1.75)--cycle;
  \draw[draw=none, fill=red, fill opacity=0.25](-0.75,-1.4)--(0.75,-1.3)--(0.75,-1.75)--(-0.75,-1.75)--cycle;
  \node at (-0.5,-0.5) [circle,fill,inner sep=1pt]{};
  \node at (0.5,-0.5) [circle,fill,inner sep=1pt]{};
  \node at (-0.5,-1.5) [circle,fill,inner sep=1pt]{};
  \node at (0.5,-1.5) [circle,fill,inner sep=1pt]{};
  \draw[line width=0.3mm](-0.5,-0.5)--(0.25, -0.5);
  \draw[line width=0.3mm](-0.5,-0.5)--(-0.5,0.7);
  \draw[line width=0.3mm](-0.5,0.7)--(0.25,0.7);
  \draw[line width=0.3mm](0.25, -0.5)--(0.25,0.7);
  \node at (-0.125,0.1){$R$};
  \end{tikzpicture}
\end{minipage}
\hspace{-5mm}
\begin{minipage}{.2\textwidth}
\begin{tikzpicture}[scale=1.4]
  \draw(-0.5,-1.5)--(-0.5,1.5);
  \draw(0.5,-1.5)--(0.5,1.5);
  \draw(-0.5,-1.5)--(0.5,-1.5);
  \draw(-0.5,-0.5)--(0.5,-0.5);
  \draw(-0.5,0.5)--(0.5,0.5);
  \draw(-0.5,1.5)--(0.5,1.5);
  \draw[dashed](-0.75,0.1)--(0.75,0.5);
  \draw[dashed](-0.75,-0.3)--(0.75,-0.1);
  \draw[draw=none, fill=blue, fill opacity=0.1](-0.75,0.1)--(0.75,0.5)--(0.75,1.75)--(-0.75,1.75)--cycle;
  \draw[draw=none, fill=red, fill opacity=0.25](-0.75,-0.3)--(0.75,-0.1)--(0.75,-1.75)--(-0.75,-1.75)--cycle;
  \node at (-0.5,0.5) [circle,fill,inner sep=1pt]{};
  \node at (-0.3,0.7){$P_1$};
  \node at (0.5,0.5) [circle,fill,inner sep=1pt]{};
  \node at (0.3,0.7){$P_2$};
  \node at (-0.5,-0.5) [circle,fill,inner sep=1pt]{};
  \node at (-0.3,-0.7){$P_3$};
  \node at (0.5,-0.5) [circle,fill,inner sep=1pt]{};
  \node at (0.3,-0.7){$P_4$};
  \end{tikzpicture}
\end{minipage}
\hspace{-5mm}
\begin{minipage}{.2\textwidth}
\begin{tikzpicture}[scale=1.4]
  \draw(-0.5,-1.5)--(-0.5,1.5);
  \draw(0.5,-1.5)--(0.5,1.5);
  \draw(-0.5,-1.5)--(0.5,-1.5);
  \draw(-0.5,-0.5)--(0.5,-0.5);
  \draw(-0.5,0.5)--(0.5,0.5);
  \draw(-0.5,1.5)--(0.5,1.5);
  \draw[dashed](-0.75,0.1)--(0.75,0.3);
  \draw[dashed](-0.75,-0.1)--(0.75,-0.7);
  \draw[draw=none, fill=blue, fill opacity=0.1](-0.75,0.1)--(0.75,0.3)--(0.75,1.75)--(-0.75,1.75)--cycle;
  \draw[draw=none, fill=red, fill opacity=0.25](-0.75,-0.1)--(0.75,-0.7)--(0.75,-1.75)--(-0.75,-1.75)--cycle;
  \node at (-0.5,0.5) [circle,fill,inner sep=1pt]{};
  \node at (0.5,0.5) [circle,fill,inner sep=1pt]{};
  \node at (-0.5,-0.5) [circle,fill,inner sep=1pt]{};
  \node at (0.5,-0.5) [circle,fill,inner sep=1pt]{};
  \draw[line width=0.3mm](-0.5,-0.8)--(0.25, -0.8);
  \draw[line width=0.3mm](-0.5,-0.8)--(-0.5,-0.5);
  \draw[line width=0.3mm](-0.5,-0.5)--(0.25,-0.5);
  \draw[line width=0.3mm](0.25, -0.8)--(0.25,-0.5);
  \node at (-0.125,-0.65){$R$};
  \end{tikzpicture}
\end{minipage}
\hspace{-5mm}
\begin{minipage}{.2\textwidth}
\begin{tikzpicture}[scale=1.4]
  \draw(-0.5,-1.5)--(-0.5,1.5);
  \draw(0.5,-1.5)--(0.5,1.5);
  \draw(-0.5,-1.5)--(0.5,-1.5);
  \draw(-0.5,-0.5)--(0.5,-0.5);
  \draw(-0.5,0.5)--(0.5,0.5);
  \draw(-0.5,1.5)--(0.5,1.5);
  \draw[dashed](-0.75,-0.6)--(0.75,0.9);
  \draw[dashed](-0.75,-0.9)--(0.75,0.6);
  \draw[draw=none, fill=blue, fill opacity=0.1](-0.75,-0.6)--(0.75,0.9)--(0.75,1.75)--(-0.75,1.75)--cycle;
  \draw[draw=none, fill=red, fill opacity=0.25](-0.75,-0.9)--(0.75,0.6)--(0.75,-1.75)--(-0.75,-1.75)--cycle;
  \node at (-0.5,0.5) [circle,fill,inner sep=1pt]{};
  \node at (0.5,0.5) [circle,fill,inner sep=1pt]{};
  \node at (-0.5,-0.5) [circle,fill,inner sep=1pt]{};
  \node at (0.5,-0.5) [circle,fill,inner sep=1pt]{};
  \draw[line width=0.3mm](-0.5,0.5)--(0.35, 0.5);
  \draw[line width=0.3mm](0.35, 0.5)--(0.35,1.35);
  \draw[line width=0.3mm](0.35,1.35)--(-0.5,1.35);
  \draw[line width=0.3mm](-0.5,1.35)--(-0.5,0.5);
  \node at (-0.125,0.925){$R_+$};
  \draw[line width=0.3mm](-0.35,-0.5)--(0.5,-0.5);
  \draw[line width=0.3mm](0.5,-0.5)--(0.5,-1.35);
  \draw[line width=0.3mm](0.5,-1.35)--(-0.35,-1.35);
  \draw[line width=0.3mm](-0.35,-1.35)--(-0.35,-0.5);
  \node at(0.125,-0.925){$R_-$};
  \end{tikzpicture}
\end{minipage}
\end{center}
\caption{Cases $6$, $7$, $8$, $9$ and $10$}
\label{node e}
\end{figure}
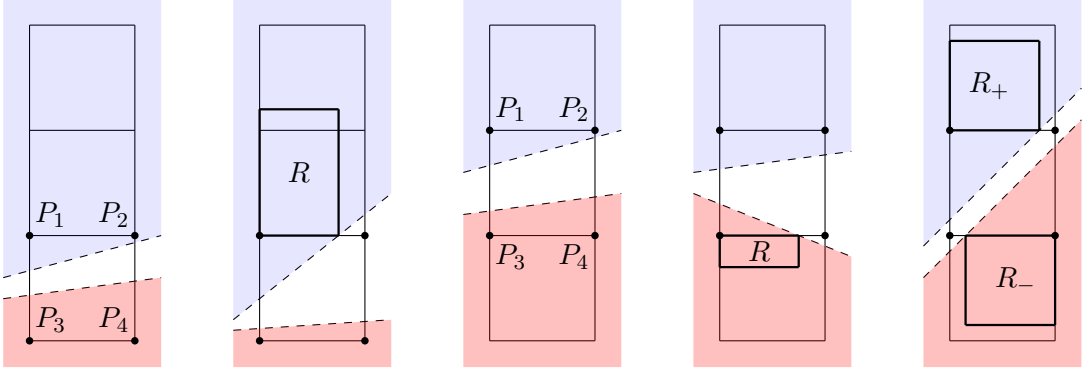

\noindent
{\bf Case 6:} If $T$ is the bottom cell and $P_1,P_2\in \cC_+$, then the two other cells are included in $\cC_+$, so
\[
|U\cap \cC|\geq 2h^2+|T\cap \cC|\geq 3|T\cap \cC| \geq 6\min(|T\cap \cC_+|,|T\cap \cC_-|).
\]
\nl
{\bf Case 7:} If $T$ is the bottom cell and $P_1\in \cC_+$ but $P_2\notin \cC_+$, $T\cap \cC_+$ is a triangle of width and height at most $h$, so there is a rectangle $R\subset (U\setminus T)\cap \cC_+$ of same width and twice as high, and thus
\[
|U\cap \cC|\geq |R|+|T\cap \cC| = 4|T\cap \cC_+|+|T\cap \cC|\geq 6\min(|T\cap \cC_+|,|T\cap \cC_-|).
\]

The same argument holds when $P_2\in \cC_+$ but $P_1\notin \cC_+$, and we necessarily have $P_1$ or $P_2$ in $\cC_+$ since $T\cap \cC_+\neq\emptyset$. If $T$ is the top cell, applying a symmetry with respect to the horizontal axis and exchanging $\cC_+$ with $\cC_-$ brings us back to Cases 6 and 7. 
\nl
\nl
{\bf Case 8:} If $T$ is the central cell,  $P_1,P_2\in \cC_+$ and  $P_3,P_4\in \cC_-$ the two other cells are included in $\cC_+$ and $\cC_-$, and we conclude as in Case 6.
\nl
\nl
{\bf Case 9:} If $T$ is the central cell,  $P_1,P_2\in \cC_+$,  $P_3\in \cC_-$ but $P_4\notin \cC_-$, the top cell is included in $\cC_+$, and there is a rectangle  $R\subset \cC_-$ of same width and height as $T\cap \cC_-$ in the bottom cell, so
\[
|U\cap \cC| \geq h^2 + |T\cap \cC|+ |R| \geq 2 |T\cap \cC|+ 2|T\cap \cC_-|\geq 6\min(|T\cap \cC_+|,|T\cap \cC_-|).
\]
The same situation occurs when only three points among $P_1,\dots,P_4$ are in $\cC$.
\nl
\nl
{\bf Case 10:} If $T$ is the central cell, only one vertex among $P_1$, $P_2$ is in $\cC_+$, and only one among $P_3$, $P_4$ is in $\cC_-$, both $T\cap \cC_+$ and $T\cap \cC_-$ are triangles, and there exist rectangles $R_+$ and $R_-$ of same widths and heights, so
\[
|U\cap \cC| \geq |R_+|+|T\cap \cC|+|R_-|\geq 3|T\cap \cC_+|+3|T\cap \cC_-|\geq 6\min(|T\cap \cC_+|,|T\cap \cC_-|).
\]

As $\cC_+$ and $\cC_-$ each contain at least one corner of $T$, we treated all cases for Node e.
\nl
\nl
{\bf Node f:} Finally, we consider the situation where there is an intersection point $X\in \Delta_u\cap\Delta_v$ in $U$, and therefore in $T$. We again match $5$ cases, illustrated in Figure \ref{node f}, depending on the position of $T$ in $U$, and of its corners with respect to $\cC$.

\begin{figure}[ht]
\begin{center}
\begin{minipage}{.2\textwidth}
\begin{tikzpicture}[scale=1.4]
  \draw(-0.5,-1.5)--(-0.5,1.5);
  \draw(0.5,-1.5)--(0.5,1.5);
  \draw(-0.5,-1.5)--(0.5,-1.5);
  \draw(-0.5,-0.5)--(0.5,-0.5);
  \draw(-0.5,0.5)--(0.5,0.5);
  \draw(-0.5,1.5)--(0.5,1.5);
  \draw[dashed](-0.75,-0.25)--(0.75,-1.6);
  \draw[dashed](-0.225,-1.75)--(0.75,0.2);
  \draw[draw=none, fill=blue, fill opacity=0.1](-0.75,-0.25)--(0.75-1.8/2.9,0.2-3.6/2.9)--(0.75,0.2)--(0.75,1.75)--(-0.75,1.75)--cycle;
  \draw[draw=none, fill=red, fill opacity=0.25](-0.225,-1.75)--(0.75-1.8/2.9,0.2-3.6/2.9)--(0.75,-1.6)--(0.75,-1.75)--cycle;
  \node at (-0.5,-0.5) [circle,fill,inner sep=1pt]{};
  \node at (0.5,-0.5) [circle,fill,inner sep=1pt]{};
  \node at (-0.5,-1.5) [circle,fill,inner sep=1pt]{};
  \node at (0.5,-1.5) [circle,fill,inner sep=1pt]{};
  \node at (0.75-1.8/2.9,0.2-3.6/2.9) [circle,fill,inner sep=1pt]{};
  \draw[line width=0.3mm](-0.75+2.5/9,-0.5)--(0.4, -0.5);
  \draw[line width=0.3mm](-0.75+2.5/9,-0.5)--(-0.75+2.5/9, 0.5827);
  \draw[line width=0.3mm](-0.75+2.5/9, 0.5827)--(0.4, 0.5827);
  \draw[line width=0.3mm](0.4, 0.5827)--(0.4, -0.5);
   \node at (-0.05,0.04){$R$};
   \node at (-0.1,-1.1){$X$};
  \end{tikzpicture}
\end{minipage}
\hspace{-7.5mm}
\begin{minipage}{.2\textwidth}
\begin{tikzpicture}[scale=1.4]
  \draw(-0.5,-1.5)--(-0.5,1.5);
  \draw(0.5,-1.5)--(0.5,1.5);
  \draw(-0.5,-1.5)--(0.5,-1.5);
  \draw(-0.5,-0.5)--(0.5,-0.5);
  \draw(-0.5,0.5)--(0.5,0.5);
  \draw(-0.5,1.5)--(0.5,1.5);
  \draw[dashed](-0.75, 1)--(0.75,-1.25);
  \draw[dashed](-0.75, -0.4)--(0.75,-0.025);
  \draw[draw=none, fill=blue, fill opacity=0.1](-0.75, 1)--(0.05,-0.2)--(0.75,-0.025)--(0.75,1.75)--(-0.75,1.75)--cycle;
  \draw[draw=none, fill=red, fill opacity=0.25](-0.75, -0.4)--(0.05,-0.2)--(0.75,-1.25)--(0.75,-1.75)--(-0.75,-1.75)--cycle;
  \node at (-0.5,-0.5) [circle,fill,inner sep=1pt]{};
  \node at (0.5,-0.5) [circle,fill,inner sep=1pt]{};
  \node at (-0.5,0.5) [circle,fill,inner sep=1pt]{};
  \node at (0.5,0.5) [circle,fill,inner sep=1pt]{};
  \node at (0.05,-0.2) [circle,fill,inner sep=1pt]{};
  \draw[dotted, line width=0.2mm](-0.75, -0.2)--(0.75,-0.2);
  \node at (-0.625, 0){$H$};
  \draw[line width=0.3mm](0.05,-0.2)--(-0.75+1/3, 1.5);
  \draw[line width=0.3mm](-0.75+1/3, 1.5)--(0.5,1.5);
  \draw[line width=0.3mm](0.05,-0.2)--(0.5,-0.2+0.1125*17/7);
  \draw[line width=0.3mm](0.5,-0.2+0.1125*17/7)--(0.5,1.5);
  \draw[line width=0.3mm](0.05,-0.2)--(-0.5,-0.2-0.1375*13/3);
  \draw[line width=0.3mm](0.05,-0.2)--(0.25,-1.5);
  \draw[line width=0.3mm](-0.5,-1.5)--(0.25,-1.5);
  \draw[line width=0.3mm](-0.5,-1.5)--(-0.5,-0.2-0.1375*13/3);
  \end{tikzpicture}
\end{minipage}
\hspace{-5mm}
\begin{minipage}{.2\textwidth}
\begin{tikzpicture}[scale=1.4]
  \draw(-0.5,-1.5)--(-0.5,1.5);
  \draw(0.5,-1.5)--(0.5,1.5);
  \draw(-0.5,-1.5)--(0.5,-1.5);
  \draw(-0.5,-0.5)--(0.5,-0.5);
  \draw(-0.5,0.5)--(0.5,0.5);
  \draw(-0.5,1.5)--(0.5,1.5);
  \draw[dashed](-0.75,0.55)--(0.75,-0.45);
  \draw[dashed](-0.75,0.25)--(0.75,-0.15);
  \draw[draw=none, fill=blue, fill opacity=0.1](-0.75,0.55)--(0,0.05)--(0.75,-0.15)--(0.75,1.75)--(-0.75,1.75)--cycle;
  \draw[draw=none, fill=red, fill opacity=0.25](-0.75,0.25)--(0,0.05)--(0.75,-0.45)--(0.75,-1.75)--(-0.75,-1.75)--cycle;
  \node at (-0.5,0.5) [circle,fill,inner sep=1pt]{};
  \node at (-0.3,0.7){$P_1$};
  \node at (0.5,0.5) [circle,fill,inner sep=1pt]{};
  \node at (0.3,0.7){$P_2$};
  \node at (-0.5,-0.5) [circle,fill,inner sep=1pt]{};
  \node at (-0.3,-0.7){$P_3$};
  \node at (0.5,-0.5) [circle,fill,inner sep=1pt]{};
  \node at (0.3,-0.7){$P_4$};
  \node at (0,0.05) [circle,fill,inner sep=1pt]{};
  \end{tikzpicture}
\end{minipage}
\hspace{-7.5mm}
\begin{minipage}{.2\textwidth}
\begin{tikzpicture}[scale=1.4]
  \draw(-0.5,-1.5)--(-0.5,1.5);
  \draw(0.5,-1.5)--(0.5,1.5);
  \draw(-0.5,-1.5)--(0.5,-1.5);
  \draw(-0.5,-0.5)--(0.5,-0.5);
  \draw(-0.5,0.5)--(0.5,0.5);
  \draw(-0.5,1.5)--(0.5,1.5);
  \draw[dashed](-0.75,0.1)--(0.75,-0.05);
  \draw[dashed](-0.75,0.7)--(0.75,-1.1);
  \draw[draw=none, fill=blue, fill opacity=0.1](-0.75,0.7)--(-0.75+6/11,0.5/11)--(0.75,-0.05)--(0.75,1.75)--(-0.75,1.75)--cycle;
  \draw[draw=none, fill=red, fill opacity=0.25](-0.75,0.1)--(-0.75+6/11,0.5/11)--(0.75,-1.1)--(0.75,-1.75)--(-0.75,-1.75)--cycle;
  \node at (-0.5,0.5) [circle,fill,inner sep=1pt]{};
  \node at (0.5,0.5) [circle,fill,inner sep=1pt]{};
  \node at (-0.5,-0.5) [circle,fill,inner sep=1pt]{};
  \node at (0.5,-0.5) [circle,fill,inner sep=1pt]{};
  \node at (-0.75+6/11,0.5/11) [circle,fill,inner sep=1pt]{};
  \draw[line width=0.3mm](-0.5,-1.4)--(0.25, -1.4);
  \draw[line width=0.3mm](-0.5,-1.4)--(-0.5,-0.5);
  \draw[line width=0.3mm](-0.5,-0.5)--(0.25,-0.5);
  \draw[line width=0.3mm](0.25, -1.4)--(0.25,-0.5);
  \node at (-0.125,-0.95){$R$};
  \node at (-0.5,0.4) [circle,fill,inner sep=1pt]{};
  \node at (-0.61,0.25){$Y$};
  \node at (0.25,-0.5) [circle,fill,inner sep=1pt]{};
  \node at (0.35,-0.3){$Z$};
  \node at (-0.3,-0.3){$P_3$};
  \end{tikzpicture}
\end{minipage}
\hspace{-5.5mm}
\begin{minipage}{.2\textwidth}
\begin{tikzpicture}[scale=1.4]
  \draw(-0.5,-1.5)--(-0.5,1.5);
  \draw(0.5,-1.5)--(0.5,1.5);
  \draw(-0.5,-1.5)--(0.5,-1.5);
  \draw(-0.5,-0.5)--(0.5,-0.5);
  \draw(-0.5,0.5)--(0.5,0.5);
  \draw(-0.5,1.5)--(0.5,1.5);
  \draw[dashed](-0.75,1.6)--(0.75,-1.7);
  \draw[dashed](-0.75,0.2)--(0.75,-0.4);
  \draw[draw=none, fill=blue, fill opacity=0.1](-0.75,1.6)--(-0.75+14/18,0.2-0.4*14/18)--(0.75,-0.4)--(0.75,1.75)--(-0.75,1.75)--cycle;
  \draw[draw=none, fill=red, fill opacity=0.25](-0.75,0.2)--(-0.75+14/18,0.2-0.4*14/18)--(0.75,-1.7)--(0.75,-1.75)--(-0.75,-1.75)--cycle;
  \node at (-0.5,0.5) [circle,fill,inner sep=1pt]{};
  \node at (0.5,0.5) [circle,fill,inner sep=1pt]{};
  \node at (-0.5,-0.5) [circle,fill,inner sep=1pt]{};
  \node at (0.5,-0.5) [circle,fill,inner sep=1pt]{};
  \draw[dotted, line width=0.2mm](-0.75, 0.2-0.4*14/18)--(0.75,0.2-0.4*14/18);
  \draw[dotted, line width=0.2mm](-0.75+14/18, -1.75)--(-0.75+14/18,1.75);
  \node at (-0.75+14/18,0.2-0.4*14/18) [circle,fill,inner sep=1pt]{};
  \draw (-0.75+14/18,0.6-0.4*14/18) arc (90:114.44:0.4);
  \node at (0.1,0.6-0.4*14/18){\scriptsize $\phi$};
  \draw (-0.35+14/18,0.2-0.4*14/18) arc (360:338.2:0.4);
  \node at (0.35,0){\scriptsize $\psi$};
  \node at (-0.25,-0.2){\scriptsize $l$};
  \node at (-0.05,-0.3){\scriptsize $z$};
  \end{tikzpicture}
\end{minipage}
\end{center}
\caption{Cases $11$, $12$, $13$, $14$ and $15$}
\label{node f}
\end{figure}
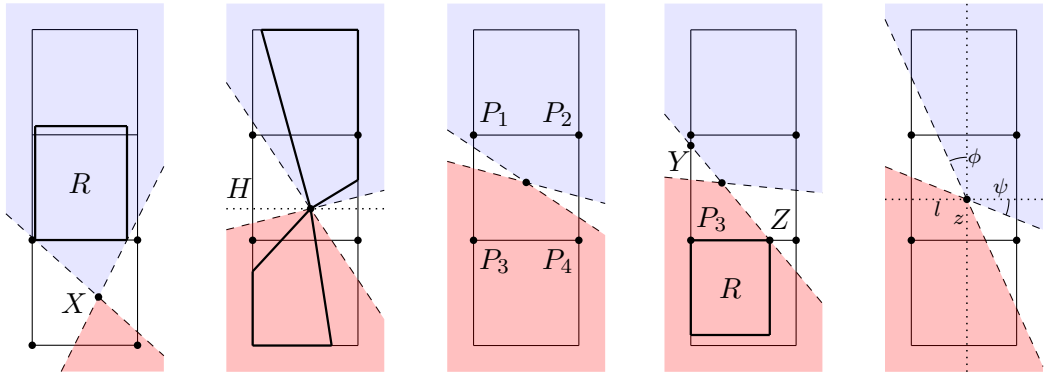

\noindent
{\bf Case 11:} If $T$ is the bottom cell, as $\Delta_u$ and $\Delta_v$ pass through the central cell of $S$, $U$ is included in the central column of $S$, and no corner of $T$ can be in $\mathring \cC_+$, since otherwise $\Delta$ would have to pass through that corner, according to the definition of the columns. As a consequence, $\Delta_u$ and $\Delta_v$ necessarily pass through the central cell of $U$, so $T\cap \cC_+$ is a triangle, and we proceed as in Case 7. The same happens if $T$ is the top cell, so in the rest of the proof we only consider situations where $T$ is the central cell.
\nl
\nl
{\bf Case 12:} If the horizontal line $H$ passing through $X$ does not intersect $\cC$ at any other point, $\cC_+$ is entirely above $H$ and $\cC_-$ entirely below. Denoting $z=X_2-\o x_2+\frac{h}{2}\in \left(0, h\right)$, the vertical dilation with respect to $H$ by a factor $\frac{2h-z}{h-z}$ sends $T\cap \cC_+$ in $U\cap \cC_+$, and the vertical dilation with respect to $H$ by a factor $\frac{h+z}{z}$ sends $T\cap \cC_-$ in $U\cap \cC_-$, so
\[
|U\cap \cC|\geq \frac{2h-z}{h-z}|T\cap \cC_+|+\frac{h+z}{z}|T\cap \cC_-|\geq 6\min(|T\cap \cC_+|,|T\cap \cC_-|)
\]
because $\frac{2h-z}{h-z}+\frac{h+z}{z}=2+\frac{h^2}{z(h-z)}\geq 6$ for $z\in (0,h)$. 

In the remaining cases, up to a symmetry with respect to the vertical axis, we can assume that $X+\R^2_+\subset \cC_+$ and $X+\R_-^2\subset \cC_-$, and in particular $P_2\in \cC_+$ and $P_3\in \cC_-$.
\nl
\nl
{\bf Case 13:} If $P_1\in \cC_+$ and $P_4\in \cC_-$, the situation is similar to Case 8.
\nl
\nl
{\bf Case 14:} If $P_1\in \cC_+$ and $P_4\notin \cC_-$, the top cell is included in $\cC_+$, and one of the lines $\Delta_u$ or $\Delta_v$ intersects the line segments $[P_1,P_3]$ and $[P_3,P_4]$ at points $Y$ and $Z$. Then the triangle $YP_3Z$ is included in $T$ and contains $T\cap \cC_-$, so there is a rectangle $R$ of same width and height in $(U\setminus T)\cap \cC_-$. In the end
\[
|U\cap \cC| \geq h^2+|T\cap \cC|+|R|\geq 2|T\cap \cC|+2|T\cap \cC_-|\geq 6\min(|T\cap \cC_+|,|T\cap \cC_-|).
\]
The same approach treats the symmetric case $P_1\notin \cC_+$ and $P_4\in \cC_-$,
\nl
\nl
{\bf Case 15:} Finally, if $P_1\notin \cC_+$ and $P_4\notin \cC_-$, denote $l=X_1-\o x_1+\frac{h}{2}\in (0,h)$, $z=X_2-\o x_2+\frac{h}{2}\in (0,h)$, $\phi\in (0,\frac\pi4)$ the angle between the vertical axis and the line among $\Delta_u$ and $\Delta_v$ that intersects $[P_1,P_2]$, and $\psi\in (0,\frac\pi4)$ the angle between the line among $\Delta_u$ and $\Delta_v$ that intersects $[P_1,P_3]$ and the horizontal axis. As $|\arg(\Delta)|\leq\frac\pi4$, $\phi\geq \psi$ so $\tan(\psi)\leq \tan(\phi)=:t \leq 1$.

We can now compute
\begin{align*}
|T\cap \cC_+|&=(h-l)(h-z)+\frac{1}{2}(h-l)^2\tan\psi+\frac{1}{2}(h-z)^2\tan\phi,\\
|T\cap \cC_-|&=lz+\frac{1}{2}l^2\tan\psi+\frac{1}{2}z^2\tan\phi,
\end{align*}
and
\[
|(U\setminus T)\cap \cC|\geq (h-l)h+(h-z)th + lh+ zth =(1+t)h^2.
\]
If $l+z\leq h$, we get
\[
|(U\setminus T)\cap \cC|\geq (1+t)(l+z)^2-(1-t)(l-z)^2= 4lz+2t(l^2+z^2)\geq 4|T\cap \cC_-|.
\]
Similarly, $l+z\geq h$ implies $|(U\setminus T)\cap \cC|\geq 4|T\cap \cC_+|$. In any case, we found
\[
|U\cap \cC|=|T\cap \cC|+|(U\setminus T)\cap \cC|\geq 6\min(|T\cap \cC_+|,|T\cap \cC_-|),
\]
which concludes the proof.

As a last remark, note that the constants $\alpha=h^{-2}$ and $\mu=\frac{3}{2}h^2$ in Proposition \ref{prop stencil}
are sharp, since equality is attained by functions of constant sign on each cell for $\alpha$, and by $w=u-v$ with $\arg(\vec n_u) \in  \frac\pi 4\mathbb Z$, $c_u=0$ and $v=u-1$ for $\mu$.

\bibliographystyle{plain}
\bibliography{references}

\end{document}